\documentclass{amsart}

\usepackage{amssymb}
\usepackage{amsfonts}
\usepackage{amsthm}
\usepackage{amsmath}
\usepackage{bbm}
\usepackage{xcolor}
\usepackage{tikz}
\usepackage{hyperref}
\usepackage{ulem}

\usepackage{hyperref}

\newtheorem{theorem}{Theorem}[section]
\newtheorem{proposition}[theorem]{Proposition}

\newtheorem{lemma}[theorem]{Lemma}

\theoremstyle{definition}
\newtheorem{definition}[theorem]{Definition}

\theoremstyle{remark}
\newtheorem{remark}{Remark}[section]

\numberwithin{equation}{section}

\newcommand{\R}{\mathbb{R}}

\newcommand{\N}{\mathbb{N}}
\newcommand{\Z}{\mathbb{Z}}
\newcommand{\C}{\mathbb{C}}

\newcommand{\pvector}[1]{
  \begin{pmatrix}
    #1
  \end{pmatrix}} 
\newcommand{\ddirac}[1]{
 \,\boldsymbol{\delta}\!\pvector{#1}\!} 
\renewcommand{\d}{\,{\rm d}}

\renewcommand{\hat}{\widehat}

\newcommand{\scriptD}{\mathcal{D}}
\newcommand{\scriptE}{\mathcal{E}}

\newcommand{\scriptS}{\mathcal{S}}

\newcommand{\jp}[1]{\langle{#1}\rangle}
\newcommand{\qtq}[1]{\:\text{#1}\:}

\DeclareMathOperator*{\wklim}{wk-lim}

\DeclareMathOperator*{\supp}{supp}

\newenvironment{CI}{\begin{list}{{\ $\bullet$\ }}{%
\setlength{\topsep}{0mm}\setlength{\parsep}{0mm}\setlength{\itemsep}{0mm}%
\setlength{\labelwidth}{8mm}\setlength{\itemindent}{9mm}\setlength{\leftmargin}{0mm}%
\setlength{\labelsep}{1mm} }}{\end{list}}

 \usepackage[shortlabels]{enumitem}
                    \setlist[enumerate, 1]{1\textsuperscript{o}}

\begin{document}

\title[Exponentials rarely maximize cone extension]{Exponentials rarely maximize Fourier extension inequalities for cones}

\author[Negro]{Giuseppe Negro}
\author[Oliveira e Silva]{Diogo Oliveira e Silva}
\author[Stovall]{Betsy Stovall}
\author[Tautges]{James Tautges}

\address{  
Center for Mathematical Analysis, Geometry and Dynamical Systems \&
Instituto Superior T\'{e}cnico\\
Av. Rovisco Pais\\ 
1049-001 Lisboa, Portugal.} 
\email{giuseppe.negro@tecnico.ulisboa.pt}
\email{diogo.oliveira.e.silva@tecnico.ulisboa.pt}

\address{University of Wisconsin--Madison\\
480 Lincoln Drive\\ 
Madison, WI 53706\\
USA.}
\email{stovall@math.wisc.edu}
\email{tautges2@math.wisc.edu}


\date{\today}

\begin{abstract}
We prove the existence of maximizers and the precompactness of $L^p$-normalized maximizing sequences modulo symmetries for all valid scale-invariant Fourier extension inequalities on the cone in $\R^{1+d}$. 
In the range for which such inequalities are conjectural, our result is conditional on the boundedness of the extension operator. 
Global maximizers for the $L^2$ Fourier extension inequality on the cone in $\R^{1+d}$ have been characterized in the lowest-dimensional cases $d\in\{2,3\}$. 
We further prove that these functions are critical points for the $L^p$ to $L^q$ Fourier extension inequality if and only if $p = 2$.
\end{abstract}

\maketitle

\section{Introduction}

In this article, we consider the Fourier extension operator on the cone,
\begin{equation}\label{eq:extension_op}
\scriptE f(t,x) := \int_{\R^d} e^{i(t,x)\cdot(|\xi|,\xi)} f(|\xi|,\xi)\, \tfrac{\d\xi}{|\xi|}, \qquad (t,x)\in\R^{1+d},
\end{equation}
initially defined on smooth functions with compact support in $\R^d \setminus\{0\}$.  Denoting by $\d\mu(\tau, \xi):=\ddirac{\tau-\lvert \xi\rvert}\lvert \xi\rvert^{-1}\d\tau \d\xi$ the unique up to normalization Lorentz-invariant measure on the cone, the restriction conjecture predicts the validity of the estimate 
\begin{equation}\label{eq:cone_restr}
    \lVert \mathcal E f\rVert_{L^q(\mathbb R^{1+d})}\le  C_{p,q}\lVert f\rVert_{L^p(\d\mu)},
\end{equation}
for all exponents $1\leq p,q\leq \infty$ satisfying
 \begin{equation}\label{eq:RCrange}
1\leq p<\frac{2d}{d-1} \text{ and } {q=\frac{d+1}{d-1}p'=:q(p).}
 \end{equation}
The variable $q$ will be used without decoration when $p$ is clear from context. 

The case $(p,q)=(1,\infty)$ of \eqref{eq:cone_restr} is elementary, and the case $(p,q)=(2,2\frac{d+1}{d-1})$ goes back to the work of Stein \cite{St93}, Strichartz \cite{St77} and Tomas \cite{To75}. The restriction conjecture for the cone has been established by Barcelo \cite{Ba85} (orthogonality) when $d=2$, by Wolff \cite{Wo01} (bilinear methods) when $d=3$, and recently by Ou--Wang \cite{OW22} (polynomial partitioning) when $d=4$. The question is open in all higher dimensions, with the current record due to Ou--Wang \cite{OW22}.  

We are interested in the {\it sharp} form of \eqref{eq:cone_restr}, and so consider the operator norm
\begin{equation}
    \label{eq:optimalC}
    {\bf A}_{p, q}:=\sup_{f\neq 0}\frac{\lVert \mathcal Ef\rVert_{L^q(\mathbb R^{1+d})}}{\lVert f\rVert_{L^p(\d\mu)}},
\end{equation}
where the supremum is taken over all nonzero $f\in L^p(\d\mu)$.
In other words, ${\bf A}_{p, q}$ is the optimal, or smallest, constant $C_{p,q}$ for which \eqref{eq:cone_restr} holds.
 It is elementary to check that ${\bf A}_{1,\infty}=1$, and that any nonnegative $f$ maximizes the corresponding inequality.
Our first main result addresses the precompactness of arbitrary maximizing sequences for \eqref{eq:cone_restr} and, in particular, the existence of maximizers.

\begin{theorem}\label{T:extremizers exist intro}
Assume that $\scriptE$ extends as a bounded linear operator from $L^{p_0}(\d\mu)$ to {$L^{q_0}(\R^{1+d})$}, for some $1 < p_0 < {2d}/(d-1)$ and $q_0:=q(p_0)$.  Then, for all $1 < p < p_0$ and $q:=q(p)$,  there exist nonzero functions $f \in L^p(\d\mu)$, such that $\|\scriptE f\|_q = {\bf A}_{p,q}\|f\|_p$.  Furthermore, if $\{f_n\} \subseteq L^p(\d\mu)$ is any norm-one sequence with $\lim_{n\to\infty} \|\scriptE f_n\|_q = {\bf A}_{p,q}$, then there exists a subsequence of $\{f_n\}$ and a sequence $\{S_n\}$ of symmetries of $\scriptE$ such that $S_n f_n$ converges in $L^p$ to a maximizer of $\scriptE$.    
\end{theorem}
\noindent The symmetries of $\scriptE$ to which we refer here are dilations, Lorentz boosts, and spacetime translations, and will be given explicitly in \S\ref{sec:existence}. 
Precompactness of maximizing sequences for \eqref{eq:cone_restr} modulo symmetries was previously established by Quilodrán \cite{Qu13}  ($d=2$) and Ramos \cite{Ra12} ($d\geq 2$), but only in the Strichartz case $p=2$. Another result along the same lines when $p=2$, but in the non-endpoint case, can be found in \cite{FVV12}. 
Analogues of Theorem \ref{T:extremizers exist intro} are known for the paraboloid \cite{St20} and the sphere \cite{FS22}.  Indeed, the proof of Theorem \ref{T:extremizers exist intro} follows the principle established in \cite{St20} of first proving that maximizing sequences possess good frequency localization, and then using the $L^2$ theory to establish a profile decomposition of frequency-localized sequences with non-negligible extensions. However, the symmetry group of the cone is more complex than that of the sphere or paraboloid, leading us to perform the  initial frequency localization in two stages, first to a single dyadic annulus, and then to a sector within that annulus. 

Once the existence of maximizers has been established, it is natural to ask whether they can be explicitly determined.
This has been done in just a few cases, which are most easily related to the present context via the following preliminary observations. 
The operator $\mathcal E$ defined in \eqref{eq:extension_op} corresponds to the half-wave equation $u_t=iDu$. 
Defining the {\it half-wave propagator} as
\begin{equation}\label{eq:HalfWaveProp}
e^{itD}g(x):=\frac{1}{(2\pi)^d} \int_{\R^d}  e^{it|\xi|} \widehat g(\xi)  e^{ix\cdot\xi}\d\xi,
\end{equation}
one readily sees that, if $\widehat g(\xi)=|\xi|^{-1}f(|\xi|,\xi)$, then 
$e^{itD}g(x)=(2\pi)^{-d}\mathcal E f(t,x)$, and 
\[\|e^{itD}g\|_{L^q(\R^{1+d})} \|\widehat g\|_{L^p(\R^d,|\xi|^{p-1}\d\xi)}^{-1}=
(2\pi)^{-d} \|\mathcal E f\|_{L^q(\R^{1+d})} \|f\|_{L^p(\d\mu)}^{-1}.\]
Consequently, \eqref{eq:cone_restr} can be recast in sharp form as 
\begin{equation}\label{eq:cone_restr_recast}
\|e^{itD}g\|_{L^q(\R^{1+d})}\leq \frac{{\bf A}_{p,q}}{(2\pi)^d} \|\widehat g\|_{L^p(\R^d,|\xi|^{p-1}\d\xi)},
\end{equation}
which is invariant under the same group of symmetries as $\mathcal E$. The Strichartz case $(p,q)=(2,2\frac{d+1}{d-1})$ reads, in sharp form,
\begin{equation}\label{eq:SteinTomas}
    \lVert e^{itD} g\rVert_{L^{2\frac{d+1}{d-1}}(\mathbb R^{1+d})} \le \frac{{\bf A}_{2,2(d+1)/(d-1)}}{(2\pi)^d}\lVert g\rVert_{\dot{H}^{1/2}(\R^d)}.
\end{equation}
We can now introduce the class of functions that are known to  maximize~\eqref{eq:cone_restr_recast} in a few cases. Henceforth we define
\begin{equation}\label{eq:Fourier_Trans_fstar}
    \widehat{g}_\star(\xi):= C_d \lvert \xi\rvert^{-1}\exp(-\lvert \xi\rvert),
\end{equation}
where $C_d>0$ is chosen in order to ensure that\footnote{This Fourier transform is  easily computed via the Gaussian superposition  \[\lvert\xi\rvert^{-1}\exp(-\lvert\xi\rvert)=\frac1{2\pi^{1/2}}\int_0^\infty \exp\left(-s-\frac{\lvert\xi\rvert^2}{4s}\right)\, \frac{\d s}{s^{3/2}}.\]} 
\begin{equation}\label{eq:Foschian_Physical} 
    g_\star(x)=\left(\frac{2}{1+\lvert x\rvert^2}\right)^{\frac{d-1}{2}}.
\end{equation}
\begin{definition}\label{def:Foschian}
        A nonzero function $f_\star$ is an $\mathfrak F$\textit{-function} if it can be obtained from $g_\star$ by applying one or more symmetries of~\eqref{eq:cone_restr_recast}; thus, 
        \begin{equation*}
            \widehat{f}_\star(\xi)= \lvert \xi\rvert^{-1}\exp(A\lvert \xi\rvert+b\cdot \xi + c),
        \end{equation*}
        for some $A,c\in\mathbb C$, $b\in\mathbb C^d$ and $|\Re(b)|<-\Re(A)$.
\end{definition}
\noindent This definition alludes to Foschi, who in~\cite[Eqs.\@ (33) and (46)]{Fo07} characterized the maximizers of~\eqref{eq:SteinTomas} in the lowest dimensional cases $d\in\{2,3\}$ as those functions from Definition~\ref{def:Foschian}; see also \cite{Ca09}.
So far these remain the only known instances of global maximizers to~\eqref{eq:SteinTomas}, and provide an affirmative partial answer to the general question of whether exponentials maximize Fourier extension for a conic section; see the recent survey \cite{NOST22}.
It is thus sensible to ask whether this is an isolated fact, or whether $\mathfrak F$-functions maximize other cases of~\eqref{eq:SteinTomas} and, more generally, of~\eqref{eq:cone_restr_recast}.

A necessary condition for an $\mathfrak F$-function $f_\star$ to maximize \eqref{eq:cone_restr_recast} is that it solves the Euler--Lagrange equation 
\begin{equation}\label{eq:EulerLagrange}
    \Re \int_{\R^{1+d}} \lvert e^{itD}f_\star \rvert^{q-2}\overline{e^{itD}f_\star} e^{itD}g\, \d t\d x 
    = \lambda_{p,q} \Re \int_{\R^d} (\lvert \widehat{f}_\star\rvert^{p-2} \overline{\widehat{f}_\star} \widehat{g})(\xi)\lvert \xi\rvert^{p-1}\, \d\xi,
\end{equation}
where the Lagrange multiplier $\lambda_{p,q}=\lambda_{p,q}(f_\star)$ does not depend on the arbitrary test distribution $g$. Here we require $\hat{g}\in L^p(\mathbb R^d, \lvert\xi\rvert^{p-1}\, \d \xi)$ and $e^{itD} g\in L^q(\mathbb R^{1+d})$. The latter condition follows immediately from the former if $e^{itD}$ defines a bounded operator in the sense of \eqref{eq:cone_restr_recast}, and it ensures that~\eqref{eq:EulerLagrange} continues to be well-defined even without such an assumption.  If~\eqref{eq:EulerLagrange} holds for all such $g$, then $f_\star$ is said to be a {\it critical point} for inequality \eqref{eq:cone_restr_recast}. We now state our second main result.
\begin{theorem}\label{thm:FoschianRarelyMaximize}
Let $d\geq 2$, let $1<p<{2d}/(d-1)$, {and set $q=q(p)$.} 
Then $\mathfrak F$-functions are critical points for the $L^p\to L^q$ inequality~\eqref{eq:cone_restr_recast} if and only if $p=2$.
\end{theorem}

\noindent Versions of Theorem \ref{thm:FoschianRarelyMaximize} have been established for paraboloids \cite{CQ14} and hyperbolic paraboloids \cite{COS22}, but the proofs are complex-analytic and tied to the fact that gaussians are entire functions. By contrast, any $\mathfrak F$-function is singular at the origin in Fourier space.


The proof of Theorem \ref{thm:FoschianRarelyMaximize} relies on the {\it Penrose transform}, a functional transform associated to a map that compactifies  $\mathbb R^{1+d}$ by conformally sending it into $[-\pi, \pi]\times \mathbb S^d$. Letting $\d\sigma$ denote the usual surface measure on the unit sphere $\mathbb S^d$, the left-hand side of~\eqref{eq:cone_restr_recast} can be expressed as 
\begin{equation}\label{eq:intro_symmetry}
    \lVert e^{itD}g\rVert_{L^q(\mathbb R^{1+d})}^q=\frac12 \int_{-\pi}^\pi \int_{\mathbb S^d} \left \lvert e^{iTD_{\mathbb S^d}} G\right\rvert^q \lvert \Omega\rvert^{(d+1)(\frac{p'}{2}-1)}\d \sigma\d T,
\end{equation}
where $g$ and $G$ are related via the Penrose transform, and $D_{\mathbb S^d}$ is the spherical fractional operator. What is especially relevant for our present discussion is the {\it conformal factor} $\Omega$, a non-constant function on $[-\pi, \pi]\times \mathbb S^d$ that acts as a symmetry breaker. Indeed, the exponent of $\Omega$ vanishes precisely when $p=2$, yielding invariance under arbitrary rotations of $\mathbb S^d$. This is a hidden symmetry of~\eqref{eq:cone_restr_recast} when $p=2$, which is the qualitative reason for $\mathfrak F$-functions not being critical points when $p\neq 2$. We highlight  that $\lvert \Omega\rvert^{(d+1)({p'}/{2}-1)}$ is integrable on $[-\pi, \pi]\times \mathbb S^d$ if and only if $p$, $q$ belong to the conjectural range~\eqref{eq:RCrange}; thus this constitutes an alternative derivation of the necessary conditions for the restriction conjecture to the cone.

The relevance of the Penrose transform to sharp restriction theory was realized in \cite{Ne18} and further explored in \cite{GN20, Ne22}.
In the Strichartz case $p=2$, when the Penrose transform extends to a surjective isometry of $H^{1/2}(\mathbb S^d)$ onto $\dot{H}^{1/2}(\mathbb R^d)$, $\mathfrak F$-functions  are known to be local maximizers (and thus critical points) for \eqref{eq:SteinTomas} in all dimensions $d\geq 2$ \cite{GN20}, but the case $p\neq 2$ does not seem to have been previously considered in the literature. 

Sharp restriction theory on the cone has a short but rich history. 
Further to the aforementioned works, we refer to  the papers \cite{BJ15,BJO16,BJOS17,BR13,Bulut} and the survey \cite{FOS17}.
It would be interesting to establish analogues of Theorems \ref{T:extremizers exist intro} and \ref{thm:FoschianRarelyMaximize} for the two-sheeted cone, even though the situation is different there as $\mathfrak F$-functions are not critical points in the Strichartz case $p=2$ whenever the dimension $d$ is even \cite{Ne18}. Ultimately, this is due to the failure of a formula analogous to~\eqref{eq:intro_symmetry} in the case of the two-sheeted cone in even spatial dimension; see §\ref{apsec:tangent}.

\subsection{Outline}
In \S\ref{sec:critical_points}, we prove Theorem \ref{thm:FoschianRarelyMaximize}. 
In \S\ref{sec:existence}, we prove Theorem \ref{T:extremizers exist intro}.
In Appendix~\ref{sec:penrose_tools}, we present the relevant background material on the Penrose transform  and expand on the  symmetry considerations related to~\eqref{eq:intro_symmetry}.
\subsection{Notation}
$\Re z$ and $\Im z$ denote the real and imaginary parts of a complex number $z\in\C$.
 The surface measure of the unit sphere $\mathbb S^{d-1}\subset \R^d$ is $|\mathbb S^{d-1}|=2\pi^{\frac{d}2}/
 \Gamma(\frac{d}2)$.
We use $X\lesssim Y$ or $Y\gtrsim X$ to denote the estimate $|X|\leq CY$ for an absolute positive constant $C$,  $X\simeq Y$ to denote the estimates $X\lesssim Y \lesssim X$, and $X\cong Y$ to denote the identity $X=CY$.
We  often require the implied constant $C$ in the above notation to depend on additional parameters, which we will indicate by subscripts (unless explicitly omitted); thus for instance $X \lesssim_j Y$ denotes an estimate of the form $|X| \leq C_jY$ for some $C_j$ depending on $j$. 





\section{Critical points}\label{sec:critical_points}

 In this section, we prove Theorem \ref{thm:FoschianRarelyMaximize}, which naturally splits into three cases: the {\it subcritical} case $1<p<2$, the {\it Strichartz} case $p=2$, and the {\it supercritical} case $2<p<2d/(d-1)$.
 In \S \ref{sec:EL}, we recall the simple derivation of the Euler--Lagrange equation \eqref{eq:EulerLagrange}. 
 In \S \ref{sec:P}, we apply the Penrose tools detailed in Appendix~\ref{sec:penrose_tools} to the right- and left-sides of \eqref{eq:EulerLagrange}, and obtain useful formulae which hold for every $1<p<2d/(d-1)$. The function $g_\star$ defined in \eqref{eq:Foschian_Physical} is a critical point for \eqref{eq:cone_restr_recast} when $p=2$; this is known from \cite[Remark~3.8]{GN20}, but it also follows from our formulae in Remark~\ref{rem:SteinTomasStrichartz_case} below. 
 In \S \ref{sec:s}, we treat the subcritical case via sign considerations. In \S \ref{sec:S},  we handle the supercritical case via asymptotic analysis.

    \subsection{The Euler--Lagrange equation}\label{sec:EL}
 Given an $\mathfrak F$-function $f_\star$, let
    \[\lambda_{p,q}(f_\star):=\|e^{itD}f_\star\|_{L^q(\R^{1+d})}^q \|\widehat f_\star\|_{L^p(\R^d,|\xi|^{p-1}\d\xi)}^{-q}.\]
    We see that $f_\star$ maximizes inequality~\eqref{eq:cone_restr_recast} if and only if the functional
    \begin{equation}\label{eq:Phi_Functional}
        \Phi_{p,q}(f):= \lambda_{p,q}(f_\star) \|\widehat f\|^q_{L^p(\R^d,|\xi|^{p-1}\d\xi)}-\|e^{itD}f\|^q_{L^q(\R^{1+d})}
    \end{equation}
    is nonnegative for every test distribution $f$ such that $\widehat{f}\in L^p(\mathbb R^d, \lvert \xi\rvert^{p-1}\, \d \xi)$ and $e^{itD}f \in L^q(\mathbb R^{1+d})$; these conditions ensure that ~\eqref{eq:Phi_Functional} is well-defined and finite. By construction, $\Phi_{p,q}(f_\star)=0$. 
    
    To compute the first variation associated to $\Phi_{p,q}$ at $f_\star$, i.e.\@ $\left.\partial_\varepsilon \Phi_{p,q}(f_\star + \varepsilon f)\right\rvert_{\varepsilon=0}$, we expand, for small $\varepsilon>0$,
    \begin{equation}\label{eq:compute_euler_lagrange}\notag
        \begin{split} 
            \|e^{itD}(f_\star+\varepsilon f)\|_q^q &=\|e^{itD}f_\star\|_q^q
            +\varepsilon q \Re \int_{\R^{1+d}}|e^{itD}f_\star|^{q-2}\overline{e^{itD}f_\star} e^{itD}f \d t \d x +o(\varepsilon),\\
            \|\widehat f_\star+\varepsilon\widehat f\|_p^q&=\|\widehat f_\star\|_p^q+\varepsilon q \|\widehat f_\star\|_p^{q-p}\Re\int_{\R^d} (|\widehat f_\star|^{p-2}\overline{\widehat f_\star}\widehat f)(\xi) |\xi|^{p-1} \d \xi+o(\varepsilon).
        \end{split}
    \end{equation}
    Here, the Landau symbols $o(\varepsilon)$ depend on $\lVert f_\star\rVert_p, \lVert f\rVert_p,\|e^{itD}f_\star\|_q, \|e^{itD}f\|_q$, and we abbreviated $\|\cdot\|_q=\|\cdot\|_{L^q(\mathbb{R}^{1+d})}$ and $\|\cdot\|_p=\|\cdot\|_{L^p(\R^d,|\xi|^{p-1}\d\xi)}$. 
    The Euler--Lagrange equation \eqref{eq:EulerLagrange} follows at once.

\begin{remark}
Any $\mathfrak F$-function $f_\star$ belongs to the same orbit as the function $g_\star$ from \eqref{eq:Foschian_Physical} under the action of the symmetry group $\mathcal S$ (see \S\ref{sec:symmetries} below). Since the functional $\Phi_{p, q}$ is $\mathcal S$-invariant, it follows that $f_\star$ is a critical point for \eqref{eq:cone_restr_recast} if and only if $g_\star$ is a critical point for \eqref{eq:cone_restr_recast}. Indeed, if $f_\star=Sg_\star$ for a given $S\in\mathcal S$, then
\begin{equation*}
    \begin{array}{cc}
        \Phi_{p,q}(f_\star + \varepsilon f)=\Phi_{p, q}(g_\star+\varepsilon g), & \text{provided }f=Sg.
    \end{array}
\end{equation*}
Therefore we will analyze the Euler--Lagrange equation~\eqref{eq:EulerLagrange} only at $g_\star$.  
\end{remark}

    \subsection{The effect of the Penrose transform}\label{sec:P}
    The strategy is to realize $g_\star$ as the Penrose transform of the constant function $\mathbf 1$ on $\mathbb S^d$, and to choose the test function $g$ as the Penrose transform of a single spherical harmonic. By compactifying the regions of integration in~\eqref{eq:EulerLagrange} via the Penrose map, this yields useful formulae which we will then explore. 
    \subsubsection{Right-hand side of \eqref{eq:EulerLagrange}}
    Via a change of variables, recalling~\eqref{eq:Fourier_Trans_fstar} and applying Plancherel's identity, the right-hand side of~\eqref{eq:EulerLagrange} can be rewritten as follows:
    \begin{equation}\label{eq:PenroseEulerLagrangeTwo}
        \begin{split}
           \textup{RHS}(\eqref{eq:EulerLagrange}) & \cong_{d,p}\Re\int_{\mathbb R^d} {e^{(1-p)\lvert \xi\rvert}} \widehat{g}(\xi)\, \d\xi \\ 
            &\cong_{d, p} \Re\int_{\mathbb R^d} g_\star\left(\frac{x}{p-1}\right) Dg(x)\, \d x.
        \end{split}
    \end{equation}
    We now introduce the Penrose transform $g=g(x)$ of an arbitrary $G=G(X_0, \vec X)$, defined via 
    \begin{equation}\label{eq:stereoproj_Cartesian}
        g(x) = (1+X_0)^\frac{d-1}{2} G(X_0, \vec X), \quad x=\frac{\vec X}{1+X_0},
    \end{equation}
    for $X=(X_0, X_1, \ldots, X_d)=(X_0, \vec X)\in \mathbb S^d\subset\R^{1+d}$ (see also Definition~\ref{def:Penrose_Transform}). Since $X_0^2+\lvert \vec X\rvert^2=1$, we infer 
    \begin{equation}\label{eq:xsquare_is_Xzero}
        \lvert x \rvert^2=\frac{1-X_0}{1+X_0}.
    \end{equation}
    Substituting~\eqref{eq:xsquare_is_Xzero} into $g_\star=2^{\frac{d-1}{2}}(1+\lvert \cdot \rvert^2)^\frac{1-d}{2}$ yields
    \begin{equation}\label{eq:dilated_Omega}
        g_\star\left(\frac{x}{p-1}\right) = (1+X_0)^{\frac{d-1}{2}}\left( \frac{2(p-1)^2}{(p-1)^2(1+X_0)+1-X_0}\right)^\frac{d-1}{2}.
    \end{equation}
    We observe that the right-hand side of~\eqref{eq:dilated_Omega} defines a \textit{zonal} function on $\mathbb S^d$ (i.e., one that depends on $X_0$ only). In particular, by considering~\eqref{eq:dilated_Omega} with $p=2$ and~\eqref{eq:stereoproj_Cartesian}, we recover that $g_\star$ is the Penrose transform of the constant function $\mathbf 1$; see Remark~\ref{rem:foschian_is_constant}.
    We make the Ansatz that the test function $g$ in~\eqref{eq:PenroseEulerLagrangeTwo} is the Penrose transform of 
    \begin{equation}\label{eq:F_is_Yk}
        G_k(X_0, \vec X)=Y_k(X_0),
    \end{equation}
    where $Y_k$ denotes a real-valued zonal spherical harmonic on $\mathbb S^d$ of degree $k\ge 2$.
    
    \begin{remark}\label{rem:no_low_degree_spherical_harmonics}
        We require an integer $k\ge 2$ because the spherical harmonics of degree zero and one correspond to symmetries of the functional $\Phi_{p,q}$ and are thus unsuitable to disprove~\eqref{eq:EulerLagrange}; see \S\ref{apsec:tangent} for a detailed discussion.
    \end{remark}
    \noindent We proceed to justify that such $g_k$ is an admissible test function as required in §\ref{sec:EL}.
    \begin{proposition}\label{rem:our_test_distrib_are_good}
      Let $d\geq 2$, let $1<p<{2d}/(d-1)$, and set $q=q(p)$. 
  If $G_k$ is given by \eqref{eq:F_is_Yk}, then its Penrose transform $g_k$ satisfies 
        \begin{equation}\label{eq:gadmissible} 
            \begin{array}{cc}
                \widehat{g}_k\in L^p(\mathbb R^d, \lvert\xi\rvert^{p-1}\, \d\xi),\,\,\, e^{itD}g_k\in L^q(\mathbb R^{1+d}).
            \end{array}
        \end{equation}
    \end{proposition}
    \begin{proof}
          To verify the first condition in \eqref{eq:gadmissible}, start by noting that
         \begin{equation}\label{eq:gDg}
            \lVert  \widehat{g} \rVert_{L^p(\mathbb R^d, \lvert\xi\rvert^{p-1}\, \d\xi)}^p=\int_{\mathbb R^d} \lvert \widehat{Dg}(\xi)\rvert^p \frac{\d\xi}{\lvert\xi\rvert},\quad \text{for any }g.
        \end{equation}
    On the other hand, from  the intertwining law (Lemma~\ref{lem:morpurgo}), identity \eqref{eq:F_is_Yk} and \eqref{eq:D_S_definition} it follows that
    \begin{equation}\label{eq:MYk}
        \begin{split}
        Dg_k(x)&=\left(k+\frac{d-1}2\right)(1+X_0)^{\frac{d+1}{2}} Y_k(X_0)\\  
        &= \left(k+\frac{d-1}2\right)\left( \frac{2}{1+\lvert x \rvert^2}\right)^{\frac{d+1}{2}} Y_k\left( \frac{1-\lvert x \rvert^2}{1+\lvert x \rvert^2}\right),
        \end{split}
    \end{equation}
    from which we infer
    \begin{equation}\label{eq:Dg_Hat}\notag
        \begin{array}{cc}
            \displaystyle \widehat{Dg_k}(\xi)\cong_{k, \nu}\int_{\mathbb R^d} (1+\lvert x \rvert^2)^{-\nu-1} Y_k\left(\frac{1-\lvert x \rvert^2}{1+\lvert x \rvert^2}\right) e^{-ix\cdot \xi}\, \d x, & \displaystyle \text{ where }\nu:=\frac{d-1}{2}.
        \end{array}
    \end{equation}
    We claim that $\lvert \widehat{Dg_k}(\xi)\rvert\lesssim_{k, \nu, m} (1+\lvert\xi\rvert)^{-m}$, for every $m\in\mathbb N$. Once this is proved, the required $L^p$ boundedness of $\widehat g$ will follow from \eqref{eq:gDg}. Since $Y_k$ is a polynomial of degree $k$ it suffices, for each $(\ell,m)\in \mathbb Z_{\ge 0}\times \N$, to establish the bound
    \begin{equation}\label{eq:bound_highest_term_Yk}
        I_{\ell, \nu}(\xi):=\left\lvert \int_{\mathbb R^d} (1+\lvert x \rvert^2)^{-\ell-\nu-1} (1-\lvert x \rvert^2)^\ell e^{-ix\cdot \xi}\, \d x\right\rvert\lesssim_{\ell, \nu, m} (1+\lvert\xi\rvert)^{-m}.
    \end{equation}
    By radiality, we can assume  $\xi=(\xi_1 , 0, \ldots, 0)$. Integration by parts then yields
    \begin{equation*}
        \begin{split}
            I_{\ell, \nu}(\xi)&=\left\lvert \int_{\mathbb R^d} (1+\lvert x \rvert^2)^{-\ell-\nu-1} (1-\lvert x \rvert^2)^\ell \frac{\partial^m}{\partial x_1^m}\left( \frac{e^{-ix_1\xi_1}}{\xi_1^m}\right)\, \d x\right\rvert \\ 
            &\le\frac{1}{\lvert \xi_1\rvert^m} \int_{\mathbb R^d} \left\lvert \frac{\partial^m}{\partial x_1^m}\left[ (1+\lvert x \rvert^2)^{-\ell-\nu-1}(1-\lvert x \rvert^2)^\ell\right]\right\rvert\, \d x.
        \end{split}
    \end{equation*}
    The latter integral is finite since 
    \begin{equation*}
        \begin{split}
            &\left\lvert \frac{\partial^m}{\partial x_1^m} \left[ (1+\lvert x \rvert^2)^{-\ell-\nu-1}(1-\lvert x \rvert^2)^\ell\right]\right\rvert \\ &\le \sum_{j=0}^m\binom{m}{j} \Big\lvert\frac{\partial^{m-j}}{\partial x_1^{m-j}} \left( 1+\lvert x \rvert^2\right)^{-\ell-\nu-1}\frac{\partial^{j}}{\partial x_1^{j}} \left( 1-\lvert x \rvert^2\right)^{\ell}\Big\rvert 
            \lesssim_{\ell,m,\nu} (1+\lvert x \rvert)^{-d-1-m}.
        \end{split}
    \end{equation*}
    Estimate~\eqref{eq:bound_highest_term_Yk} follows since $I_{\ell,\nu}(0)\lesssim_{\ell,\nu} 1$, and the first condition in \eqref{eq:gadmissible} is thus established. The second condition in \eqref{eq:gadmissible} follows at once from \eqref{eq:intro_symmetry} (which is further discussed in \S\ref{apsec:tangent}), and this completes the proof of the proposition.
    \end{proof}
    From \eqref{eq:PenroseEulerLagrangeTwo}, \eqref{eq:dilated_Omega}, the first identity in~\eqref{eq:MYk} and $(1+X_0)^d\d x=\d\sigma$, we obtain 
\begin{equation}\label{eq:RHSPenrose}
        \mathrm{RHS}(\eqref{eq:EulerLagrange})\cong_{d,p}\left(k+\frac{d-1}{2}\right)\int_{\mathbb S^d} \left( \tfrac{2 (p-1)^2}{(p-1)^2(1+X_0)+1-X_0}\right)^\frac{d-1}{2} Y_k(X_0)\, \d\sigma.
    \end{equation}
    If $p=2$, then the latter integral  reduces to  $\int_{\mathbb S^d} Y_k\, \d\sigma$, which necessarily vanishes since $k> 0$. We analyze this integral for $p\ne 2$ in Propositions \ref{prop:signRHS} and \ref{prop:RHSAsymptotics} below.

\subsubsection{Left-hand side of~\eqref{eq:EulerLagrange}}
    For $p,q$ in the conjectured range of boundedness~\eqref{eq:RCrange}, define the exponent $\gamma_p$ as 
    \begin{equation}\label{eq:alpha_p_q_first}
         \gamma_p:=(d+1)\left(\frac{p'}2-1\right)\in (-1, \infty).
        \end{equation}
    By Lemma~\ref{lem:LHSPenroseAppendix} and the change of variable $s=\cos R$, the left-hand side of \eqref{eq:EulerLagrange} equals
    \begin{equation}\label{eq:LHSPenrose}
        \begin{split}
         \frac{\lvert \mathbb S^{d-1}\rvert}2 \int_{-\pi}^\pi \cos(kT)\int_{-1}^{1} Y_k(s) \lvert\cos T + s\rvert^{\gamma_p} (1-s^2)^{\frac{d-2}2}\, \d s \d T.
        \end{split}
    \end{equation}
    \begin{remark}[see~\cite{GN20}]\label{rem:SteinTomasStrichartz_case}
        If $p=2$, then  $\gamma_2=0$, and the integral in~\eqref{eq:LHSPenrose} reduces to 
        \begin{equation*}
            \int_{-\pi}^\pi\int_{-1}^1 \cos(kT)Y_k(s)(1-s^2)^\frac{d-2}{2}\, \d s \d T, 
        \end{equation*}
        which clearly vanishes for every $k> 0$. In this case, we already observed in the line after~\eqref{eq:RHSPenrose} that the right-hand side of the Euler--Lagrange equation~\eqref{eq:EulerLagrange} vanishes as well. Thus~\eqref{eq:EulerLagrange} holds for every $k\geq 0$ (the case $k=0$ being immediate). The same analysis applies to a {\it general} (i.e., not necessarily zonal) spherical harmonic $Y_k=Y_k(X)$, in which case the left- and right-hand sides of~\eqref{eq:EulerLagrange} read as follows:
        \begin{equation*}
            \begin{array}{cc}   
                \int_{-\pi}^\pi\int_{\mathbb S^d} \cos(kT)Y_k(X)\, \d \sigma\d T, & \left(k+\frac{d-1}{2}\right)\int_{\mathbb S^d} \left( \tfrac{2 (p-1)^2}{(p-1)^2(1+X_0)+1-X_0}\right)^\frac{d-1}{2} Y_k(X)\, \d\sigma,
            \end{array}
        \end{equation*}
        up to irrelevant positive constants. Thus, in the case $p=2$, equation~\eqref{eq:EulerLagrange} holds when $g$ is the Penrose transform of an arbitrary spherical harmonic. In this case, the condition $\widehat{g}\in L^2(\mathbb R^d, \lvert\xi\rvert\, \d\xi)$ reduces to $g\in \dot{H}^{1/2}(\mathbb R^d)$, where $\dot H^{1/2}$ denotes the usual homogeneous Sobolev space. By~\cite[Theorem~3.3]{GN20}, the Penrose transform extends to a surjective isometry $H^{1/2}(\mathbb S^d)\to \dot{H}^{1/2}(\mathbb R^d)$, and spherical harmonics form a complete orthonormal system of the former. Thus, by linearity and density, the Euler--Lagrange equation holds when $p=2$ for every admissible $g$. This concludes the brief analysis of the Strichartz case.
    \end{remark} 
We proceed to analyze the case $p\neq 2$. Changing variables $t=-\cos T$ in~\eqref{eq:LHSPenrose}, the left-hand side of \eqref{eq:EulerLagrange} with $g=g_k$, which we denote by LHS$(\eqref{eq:EulerLagrange}, k, p)$, is seen to equal
\begin{equation}\label{eq:LHSsubcritical}\notag
    \begin{split}
        (-1)^k\lvert\mathbb S^{d-1}\rvert\int_{-1}^1\int_{-1}^1 T_k(t)(1-t^2)^{-\frac12}Y_k(s)(1-s^2)^\frac{d-2}{2} \lvert t-s\rvert^{\gamma_p}\, \d s \d t;
    \end{split}
\end{equation}
here $T_k(t):=\cos(k\arccos t)$ denotes the  Chebyshev polynomial of the first kind of degree $k$, which satisfies $T_k(-t)=(-1)^k T_k(t)$. Given $\alpha>-\frac12$,  define the functions
\begin{equation}\label{eq:f_nu}
    h_k^\alpha(t):=
    \begin{cases} 
        C_k^\alpha(t)(1-t^2)^{\alpha-\frac12}\mathbf1_{\lvert t\rvert\le 1},& \text{ if } \alpha\neq 0, \\ 
        T_k(t)(1-t^2)^{-\frac12}\mathbf1_{\lvert t \rvert \le 1}, & \text{ if } \alpha=0,
    \end{cases}    
\end{equation}
where $C_k^\alpha$ denotes the Gegenbauer polynomial of degree $k$, defined in terms of its generating function by
\begin{equation}\label{eq:Gegenbauer}
    (1-2rt +r^2)^{-\alpha}=\sum_{k=0}^\infty C_k^\alpha(t)r^k.
\end{equation}
The Gegenbauer polynomials $\{C_k^\alpha(t)\}_{k=0}^\infty$ are orthogonal in the interval $[-1, 1]$ with respect to the measure $(1-t^2)^{\alpha-1/2} \d t$, and $T_k=\frac{k}{2}C_k^0$. Henceforth we abuse notation slightly by letting 
\begin{equation}\label{eq:Yk_is_Ck} \notag
    \begin{array}{cc}
        Y_k=C_k^\nu, &  \text{with } \nu=\frac{d-1}{2}.
    \end{array}
\end{equation}
We then have that
\begin{equation}\label{eq:LHSsubcritical_convolution}
   \mathrm{LHS}(\eqref{eq:EulerLagrange}, k, p)= (-1)^k\lvert\mathbb S^{d-1}\rvert\int_{-\infty}^\infty h_k^0(t) (h_k^\nu\ast \lvert\cdot\rvert^{\gamma_p})(t)\, \d t,
\end{equation}
and the latter integral can be computed in terms of Bessel functions. With this purpose in mind, we introduce the following quantities: 
\begin{equation}\label{eq:Rodrigues_Homogeneous}
    \begin{split}     R_k^\alpha&:=\frac{\Gamma(\alpha+\frac12)\Gamma(k+2\alpha)}{2^k k!\Gamma(2\alpha)\Gamma(\alpha+k+\frac12)}, \text{ if } \alpha>-\tfrac12, \,\alpha\neq 0,\\      R_k^0&:=\frac{\sqrt{\pi}}{2^k\Gamma(k+\frac12)},\quad\quad\,\,\,
H_\gamma:=\frac{2^{\frac{\gamma+1}{2}}}{\pi^\frac\gamma2} \frac{\Gamma(\frac{\gamma+1}{2})}{ \Gamma(-\frac{\gamma}{2})}, \text{ if } \gamma>-1.
    \end{split}
\end{equation}
Note that $H_\gamma=0$ whenever $\gamma$ is a nonnegative even integer since the reciprocal Gamma function $\frac1{\Gamma}$ is entire and vanishes on $\{0,-1,-2,\ldots\}$.
\begin{lemma}\label{lem:finite_fourier_transforms}
    Let $\alpha>-\frac12$.
    The Fourier transform of the function $h_k^\alpha$ defined in \eqref{eq:f_nu} is given by
    \begin{equation}\label{eq:precise_finite_fourier_transform}\notag
        \widehat{h^\alpha_k}(\tau)
        =2^{\alpha+k} \Gamma(\alpha+k+\tfrac12)\sqrt{\pi}R_k^\alpha (-i\tau)^k \frac{J_{\alpha+k}(\lvert\tau\rvert)}{\lvert\tau\rvert^{\alpha+k}}.
    \end{equation}
\end{lemma}
\begin{proof} 
     The formula of Rodrigues \cite[p.~22, Lemma 4]{Mu98} states that, for $\alpha\ne 0$,
    \begin{equation}\label{eq:general_rodrigues}\notag
    C_k^\alpha(t)=\frac{(-1)^kR_k^\alpha}{(1-t^2)^{\alpha-\frac12}}\frac{\d^k}{\d t^k}\left( (1-t^2)^{k+\alpha-\frac12}\right).
    \end{equation}
    If $\alpha=0$, then we instead have that
    \begin{equation}\label{eq:Chebyshev_rodrigues}\notag
          T_k(t)=\frac{(-1)^kR_k^0}{(1-t^2)^{-\frac12}}\frac{\d^k}{\d t^k}\left( (1-t^2)^{k-\frac12}\right).
    \end{equation}
By the Poisson representation of Bessel functions \cite[Ch.~II, 2.3(3)]{Wa44},
     \begin{equation}\label{eq:Poisson_repr}\notag
        \frac{J_{\alpha+k}(\lvert \tau\rvert)}{\lvert \tau \rvert^{\alpha+k}}=\frac{1}{2^{\alpha+k}\Gamma(\alpha+k+\frac12)\sqrt{\pi}} \int_{-1}^1 (1-t^2)^{k+\alpha-\frac12}e^{-it\tau}\, \d t
     \end{equation}
     for $\tau\in \mathbb R$, as long as $ \Re(\alpha+k)>-\frac12$. Partial integration then yields
     \begin{equation}\label{eq:final_poisson_comp}\notag
        \begin{split}
           \widehat{h^\alpha_k}(\tau)= \int_{-\infty}^\infty h_k^\alpha(t)e^{-it\tau}\, \d t &= (-1)^kR_k^\alpha\int_{-1}^1\frac{\d^k}{\d t^k}\left( (1-t^2)^{k+\alpha-\frac12}\right) e^{-it\tau}\,\d t \\ 
            &=R_k^\alpha(-i\tau)^k\int_{-1}^1(1-t^2)^{k+\alpha-\frac12}e^{-it\tau}\, \d t \\ &=2^{\alpha+k} \Gamma(\alpha+k+\frac12)\sqrt{\pi}R_k^\alpha (-i\tau)^k \frac{J_{\alpha+k}(\lvert\tau\rvert)}{\lvert\tau\rvert^{\alpha+k}},        
        \end{split}
    \end{equation}
    as desired.
    This concludes the proof of the lemma.
\end{proof}
For $-1<\gamma\notin 2\mathbb N_0$, we will also use the Fourier transform
\begin{equation}\label{eq:FT_Homogeneous}
    \int_{-\infty}^\infty \widehat{\lvert\cdot\rvert^\gamma}(\tau) \varphi(\tau)\, \d \tau:=\int_{-\infty}^\infty |t|^\gamma \widehat{\varphi}(t) \d t = H_\gamma
    \int_{-\infty}^\infty |\tau|^{-1-\gamma} \varphi(\tau) \d \tau,
\end{equation}
the expressions being valid for any Schwartz function $\varphi$ such that $\lvert\cdot\rvert^{-1-\gamma}\varphi\in L^1(\mathbb R)$. Note that the sign of the quantity $H_\gamma\cong_\gamma \Gamma(\frac{\gamma+1}2)/\Gamma(\frac{-\gamma}2)$ defined in \eqref{eq:Rodrigues_Homogeneous} equals $(-1)^{\lfloor\frac{\gamma}2\rfloor+1}$ whenever $0<\gamma\notin2\mathbb N$.
From \eqref{eq:LHSsubcritical_convolution}, Plancherel's identity and \eqref{eq:FT_Homogeneous}, we then conclude that
\begin{equation}\label{eq:LHSafterPenrose}
    \begin{split}
        \textup{LHS}(\eqref{eq:EulerLagrange}, k, p)&\cong_d (-1)^k \int_{-\infty}^\infty \overline{\widehat{h_k^0}(\tau)}\widehat{h_k^\nu}(\tau)\widehat{\lvert\cdot\rvert^{\gamma_p}}(\tau)\, \d\tau \\ 
        &\cong_d(-1)^kH_{\gamma_p} \frac{\Gamma(k+2\nu)}{k!}\int_0^\infty \, \frac{(J_k J_{\nu+k})(\tau)}{\tau^{1+{\gamma_p}+\nu}}\,\d\tau, \\
    \end{split}
\end{equation}
   whenever $\gamma_p\neq 0$ is not a positive even integer.
   
\subsection{The subcritical case}\label{sec:s}
We begin by determining the sign of the right-hand side of the Euler--Lagrange equation, which we computed in~\eqref{eq:RHSPenrose} and henceforth denote by $\mathrm{RHS}(\eqref{eq:EulerLagrange},k, p)$.
\begin{proposition}\label{prop:signRHS}
    Let $k\ge 2$ be an integer. Then $\mathrm{RHS}(\eqref{eq:EulerLagrange},k, p)$ is positive for all $k$ if $1<p<2$, and has the sign $(-1)^k$ if $2<p<2d/(d-1)$. 
\end{proposition}
\noindent    The proof of Proposition \ref{prop:signRHS} relies on a particular consequence of the formula of Rodrigues which can be found in \cite[p.~23]{Mu98}. 
    \begin{lemma}[\cite{Mu98}]\label{lem:Rodrigues}
        Let $\varphi:[-1, 1]\to\R$ be $k$ times continuously differentiable, and let $Y_k=Y_k(X_0)$ denote a zonal spherical harmonic of degree $k$ on $\mathbb S^d$. Then
    \begin{equation}\label{eq:Rodrigues_conseq}\notag
        \int_{-1}^1 \varphi(t)\frac{Y_k(t)}{Y_k(1)}(1-t^2)^\frac{d-2}{2}\, \d t = \frac1{2^k} \frac{\Gamma(\tfrac d 2)}{\Gamma(k+\tfrac d 2)} \int_{-1}^1  \frac{\d^k \varphi}{\d t^k}(t)(1-t^2)^{k+\frac{d-2}{2}}\, \d t.
    \end{equation}
    \end{lemma}
    \begin{proof}[Proof of Proposition \ref{prop:signRHS}]
        Letting
    \begin{equation}\label{eq:a_and_b_coeff}
        a_p:=(p-1)^2-1,\,\,\, b_p:=(p-1)^2+1
    \end{equation}
    and applying Lemma \ref{lem:Rodrigues}, we see from~\eqref{eq:RHSPenrose} that
    \begin{equation}\label{eq:RHSPenroseCont}\notag
        \mathrm{RHS}(\eqref{eq:EulerLagrange},k, p)\cong_{d,p,k} \int_{-1}^1
        \frac{\d^{k}}{\d t^{k}}(a_p t+b_p)^{\frac{1-d}{2}}
        (1-t^2)^{k+\frac{d-2}{2}}\, \d t.
    \end{equation}
    In the admissible range of $p$, note that $a_p=0$ if and only if $p=2$; in that case, the integral vanishes as we have already observed. 
    Since $b_p>\lvert a_p \vert $, binomial expansion reveals that
    \begin{equation*}
    \mathrm{RHS}(\eqref{eq:EulerLagrange},k, p)\cong_{d,p,k}b_p^{\frac{1-d}{2}} \sum_{m=k}^\infty\! \binom{\frac{1-d}{2}}{m}\!
    \prod_{\ell=0}^{k-1}(m-\ell)
    \left(\frac{a_p}{b_p}\right)^m\!\!\! \int_{-1}^1 t^{m-k}(1-t^2)^{k+\frac{d-2}{2}}\, \d t.
    \end{equation*}
    By parity considerations, only the terms with even $m-k$ yield nonzero integrals. The corresponding binomial coefficient has the same sign $(-1)^m=(-1)^k$. The term $(a_p/b_p)^m$ has the sign $(-1)^k$ if $a_p<0$, i.e., if $1<p<2$; and it is positive if $a_p>0$, i.e., if $2<p<2d/(d-1)$. The result follows.
    \end{proof}

    We proceed to analyze the left-hand side of the Euler--Lagrange equation, and rely on the following result which is a particular case of~\cite[Ch.~XIII, 13.41(2)]{Wa44}. 
\begin{lemma}[\cite{Wa44}]\label{lem:Watson}
    For $\mu, \nu, \lambda\in \mathbb C$ such that $\Re(\mu+\nu+1)>\Re(\lambda)>0$, 
    \begin{equation}\label{eq:WatsonIntegral}\notag
        \int_0^\infty \,\frac{(J_\mu J_\nu)(\tau)}{\tau^\lambda}\,\d \tau=\frac{\Gamma(\lambda)}{2^\lambda}\frac{\Gamma(\frac{\mu+\nu-\lambda+1}{2})}{\Gamma(\frac{\lambda+\mu+\nu+1}{2})\Gamma(\frac{\lambda+\nu-\mu+1}{2})\Gamma(\frac{\lambda+\mu-\nu+1}{2})}.
    \end{equation}
\end{lemma}
\noindent    Before stating our next result, recall the definition \eqref{eq:alpha_p_q_first} of $\gamma_p$. 
\begin{proposition}\label{prop:main_subcritical}
    Let $k\ge 2$ be an integer. Then $\textup{LHS}(\eqref{eq:EulerLagrange}, k, p)<0$ if  $\gamma_p\in (2k-4, 2k-2)$ and $\textup{LHS}(\eqref{eq:EulerLagrange}, k, p)=0$ if $\gamma_p\in \{2k-4, 2k-2\}$.
\end{proposition}
\begin{proof}
    If $2k-4<\gamma_p< 2k-2$, then \eqref{eq:LHSafterPenrose} yields
    \begin{equation}\label{eq:LHSafterPenrose_subcritical}
        \mathrm{LHS}(\eqref{eq:EulerLagrange}, k, p)\cong_{k,d} (-1)^k H_{\gamma_p}\int_{0}^\infty  \, \frac{(J_kJ_{\nu+k})(\tau)}{\tau^{1+\gamma_p+\nu}}\,\d\tau,
    \end{equation}
    and in this case the sign of $H_{\gamma_p}$ is $(-1)^{\lfloor \frac{\gamma_p}{2}\rfloor+1}=(-1)^{k+1}$, as we noted immediately after \eqref{eq:FT_Homogeneous}. On the other hand, the positive numbers $(k, \nu+k, 1+\gamma_p+\nu)$ form a triangular triple, that is, the sum of any two terms is larger than the remaining term. Consequently, the integral in~\eqref{eq:LHSafterPenrose_subcritical} is nonnegative in light of Lemma~\ref{lem:Watson}, and this establishes the first claim. 
    If $\gamma\in\{2k-4,2k-2\}$, then the integral on the right-hand side of \eqref{eq:LHSafterPenrose_subcritical} continues to be finite. In that case, LHS$(\eqref{eq:EulerLagrange},k,p)$ is given by \eqref{eq:LHSPenrose}, which defines a continuous function of $\gamma_p$.
    Since $H_\gamma\to 0$ as $\gamma$ approaches any nonnegative even integer, taking the limits $\gamma_p\downarrow (2k-4)$ and $\gamma_p\uparrow (2k-2)$ in~\eqref{eq:LHSafterPenrose_subcritical} yields the second claim by continuity.
\end{proof}
\noindent By Proposition~\ref{prop:signRHS}, we know that RHS$(\eqref{eq:EulerLagrange}, k, p)>0$ for every $k\geq 2$ and  $1<p<2$. By Proposition~\ref{prop:main_subcritical}, we then conclude, for every $1<p<2$, that there exists $k\geq 2$ such that ~\eqref{eq:EulerLagrange} fails. This concludes the analysis of the subcritical case. 

\subsection{The supercritical case}\label{sec:S}
 We continue to consider zonal spherical harmonics $Y_k=C_k^\nu$,  the latter denoting the Gegenbauer polynomial of degree $k\geq 2$, and $\nu=\frac{d-1}2$. By  definition~\eqref{eq:a_and_b_coeff} of $a_p, b_p$,  identity~\eqref{eq:RHSPenrose} reads
    \begin{equation}\label{eq:RHSGegenbauer}\notag
        \mathrm{RHS}(\eqref{eq:EulerLagrange}, k, p)\cong_{d,p}(k+\nu)\int_{-1}^1 (a_p t + b_p)^{-\nu} C_k^\nu(t)(1-t^2)^{\nu-\frac12}\, \d t,
    \end{equation}
    which can be recognized as a certain coefficient in a Gegenbauer expansion, and estimated as follows.
    \begin{proposition}\label{prop:RHSAsymptotics}
        Let $2<p<\frac{2d}{d-1}$ and $k\geq 2$. Then
        \begin{equation}\label{eq:GegenbauerAsymptotics}
            \left\lvert \mathrm{RHS}(\eqref{eq:EulerLagrange}, k, p)\right\rvert \lesssim_{d,p}\left(\frac23\right)^{k}k^{\nu}.
        \end{equation}
    \end{proposition}    
        \noindent The proof of Proposition \ref{prop:RHSAsymptotics} relies on  \cite[Theorem 4.3]{Wa16} which has recently played a role in sharp restriction theory \cite{CNOS21}.
        We recall it for the convenience of the reader. 
        \begin{lemma}[\cite{Wa16}]\label{lem:Wang}
        Let $\alpha>0$ and $\rho>1$.
            Let $\mathcal K$ be a function that is analytic inside and on the  ellipse 
            \begin{equation*}
                \mathcal E_\rho:=\left\{s\in\C\,: s=\tfrac12\big(\rho e^{i\theta}+\rho^{-1} e^{-i\theta}\big)\,, \, 0\leq \theta\leq 2\pi\right\},
            \end{equation*}
and $ \mathcal K(t)=\sum_{k=0}^\infty a_k^\alpha C_k^\alpha(t)$ be the corresponding Gegenbauer expansion  for $t\in[-1,1]$. 
Then, for any $k\geq 1$, 
          \begin{equation}\label{eq:WangConclusion}\notag
                \lvert a_k^\alpha\rvert \lesssim_{\rho, \alpha} \left(\max_{s\in\mathcal E_\rho} |\mathcal K(s)|\right) k^{1-\alpha}\rho^{-k-1}.
            \end{equation}
        \end{lemma}
        
            \begin{proof}[Proof of Proposition \ref{prop:RHSAsymptotics}]
        The function $\mathcal K_{p, \nu}(s):=(a_p s+b_p)^{-\nu}=\sum_{k=0}^\infty a_k^\nu C_k^\nu(s)$ is analytic in the open disk centered at the origin of radius
        \begin{equation}\label{eq:complex_disk_radius}
            \left\lvert \frac{b_p}{a_p}\right\rvert =  \frac{(p-1)^2+1}{|(p-1)^2-1|
            }. 
        \end{equation}
        The right-hand side of \eqref{eq:complex_disk_radius} defines a decreasing function of $p\in (2, \frac{2d}{d-1})$, and so 
        \begin{equation}\label{eq:complex_disk_uniform}\notag
            \inf_{d\ge 2}\inf_{2<p<\frac{2d}{d-1}}\left\lvert \frac{b_p}{a_p}\right\rvert = \inf_{d\ge 2} \frac{d^2+1}{2d}=\frac54.
        \end{equation}
        We conclude that $\mathcal K_{p, \nu}$ is analytic in the open disk of radius $\frac54$, for every $2<p<2d/(d-1)$ and $\nu\ge \frac12$. Moreover, that disk contains the ellipse $\mathcal E_{\frac32}$. Lemma~\ref{lem:Wang} thus yields the estimate 
        \begin{equation}\label{eq:a_k_nu_bound}
            \lvert a_k^\nu\rvert \lesssim_{\nu}\left(\frac23\right)^{k}k^{1-\nu}.
        \end{equation}
        By orthogonality of the Gegenbauer polynomials $C_k^\nu$ with respect to the weight $w_\nu(t)=(1-t^2)^{\nu-\frac12}$, the  coefficient $a_k^\nu$ is given by
        \begin{equation*}
            a_k^\nu= 
            \lVert C_k^\nu\rVert_{L^2(w_\nu)}^{-2}
            \int_{-1}^1 (a_p t + b_p)^{-\nu} C_k^\nu(t)(1-t^2)^{\nu-\frac12} \d t.
        \end{equation*}
         The desired~\eqref{eq:GegenbauerAsymptotics} then follows from \eqref{eq:a_k_nu_bound} and the easy estimate
    \begin{equation}\label{eq:L2normGeg}\notag
        \begin{split}
            \lVert C_k^\nu\rVert_{L^2(w_\nu)}^{2}=\int_{-1}^1 C_k^\nu(t)^2\,(1- t^2)^{\nu - \frac12}\,\d t &=\frac{2^{1-2\nu}\,\pi}{\Gamma(\nu)^2}\frac{\Gamma(k + 2\nu)}{k!\, (k + \nu)} 
            \lesssim_{\nu} k^{2\nu-2}, 
        \end{split}
    \end{equation} 
    where we used the classical asymptotic for the Gamma function \cite{We48},
\begin{equation*}
    \begin{array}{cc}
        \displaystyle\lim_{k\to \infty}\frac{\Gamma(a+k)}{\Gamma(b+k)} k^{b-a}=1, \text{ for }a, b\in\mathbb R.
    \end{array}\qedhere
\end{equation*} 
    \end{proof}
    
 The sign considerations of Propositions~\ref{prop:signRHS} and ~\ref{prop:main_subcritical} do not suffice for the analysis of the supercritical case since both sides of the Euler--Lagrange equation~\eqref{eq:EulerLagrange} then seem to have the same sign. We thus resort to the study of the precise asymptotic behaviour of LHS$(\eqref{eq:EulerLagrange}, k, p)$, as $k\to \infty$.

\begin{proposition}\label{prop:LHSasympt}
Let $2<p<\frac{2d}{d-1}$ and $k\geq 2$.
Then:
 \begin{equation}\label{eq:LHSAsymptotics}\notag
       \lvert \textup{LHS}(\eqref{eq:EulerLagrange},k, p)\rvert\cong_{p, d} \frac{ \Gamma(k-\frac{\gamma_p}2)\Gamma(k+2\nu)}{\Gamma(\frac{\gamma_p}{2}+\nu+k+1)\Gamma(k+1)}.
 \end{equation}
\end{proposition}
\begin{proof}
This follows from \eqref{eq:LHSafterPenrose} and Lemma \ref{lem:Watson} at once.
\end{proof}
\noindent By Propositions~\ref{prop:RHSAsymptotics} and \ref{prop:LHSasympt} we have, for every $2<p<\frac{2d}{d-1}$ and $k\geq 2$, 
\begin{align}\label{eq:RatioAsymptotic}
        \left\lvert \frac{\textup{RHS}(\eqref{eq:EulerLagrange},k, p)}{\textup{LHS}(\eqref{eq:EulerLagrange},k, p)}\right\rvert &\lesssim_{d,p}\left(\frac23\right)^k k^\nu \frac{\Gamma(\frac{\gamma_p}{2}+\nu+k+1)\Gamma(k+1)}{  \Gamma(k-\frac{\gamma_p}2)\Gamma(k+2\nu)}\notag\\
        &\lesssim_{d,p} \left(\frac23\right)^k k^{\gamma_p+2}.
    \end{align}
Since~\eqref{eq:RatioAsymptotic} tends to $0$, as $k\to \infty$, it follows that, for each  $p$ in the supercritical range, there exists $k\geq 2$ such that the test function corresponding to $Y_k$ does not satisfy the Euler--Lagrange equation~\eqref{eq:EulerLagrange}. This completes the proof of Theorem~\ref{thm:FoschianRarelyMaximize}.

\section{Existence of maximizers}\label{sec:existence}



In this section, we prove Theorem~\ref{T:extremizers exist intro}.  In doing so, it will be convenient to identify the cone with $\R^d$ via the projection $(|\xi|,\xi) \mapsto \xi$.  We will thus abuse notation by writing $\d\mu(\xi) = \tfrac{\d\xi}{|\xi|}$ and
$$
\scriptE f(t,x) := \int_{\R^d} e^{i(t,x)\cdot(|\xi|,\xi)}f(\xi)\,\tfrac{\d\xi}{|\xi|}.
$$

Throughout this section, we will say that an object (such as a constant) is \textit{permissible} if it depends on $d,p,p_0$, and an upper bound for the operator norm of $\scriptE:L^{p_0}(\textup \d\mu) \to L^{q_0}(\R^{1+d})$, {where $q_0:=q(p_0)$} as in \eqref{eq:RCrange}, and implicit constants are required to be permissible in this sense.

We begin with a discussion of the key symmetries for our analysis.

\subsection{Symmetries}\label{sec:symmetries}

By a \textit{symmetry} of $\scriptE$, we mean an isometry $S$ of $L^p(\frac{\textup d \xi}{|\xi|})$ for which there exists an isometry $T$ of $L^q(\R^{1+d})$ such that $\scriptE \circ S = T \circ \scriptE$, when restricted to the Schwartz class.  

The following symmetries (and their compositions) play a particularly important role in our analysis:  
\begin{CI}
\item Conic dilations:
$$f(\xi) \rightsquigarrow \lambda^\frac{d-1}{p}f(\lambda \xi), \qquad \scriptE f (t,x) \rightsquigarrow \lambda^{-\frac{d+1}{q}} \scriptE f(\lambda^{-1}t, \lambda^{-1}x),
$$
for $\lambda > 0$; 
\item Lorentz boosts:  
\begin{gather*}
f(\xi) \rightsquigarrow f(\xi^\perp+\jp{\xi_0}\xi^\parallel-|\xi|\xi_0),\\
\scriptE f(t,x) \rightsquigarrow \scriptE f(\jp{\xi_0}t+x\cdot\xi_0,x^\perp+\jp{\xi_0}x^\parallel+t \xi_0),
\end{gather*}
for $\xi_0 \in \R^d$, 
where the parallel and perpendicular parts of $\xi,x$ are taken with respect to $\xi_0$ and $\jp{\xi}:=\sqrt{1+|\xi|^2}$;
\item Sectorial expansions, obtained by composing a Lorentz boost with a dilation:
\begin{gather*}
f(\xi) \rightsquigarrow
\lambda^{\frac{d-1}p}f(\tfrac{1-\lambda^2}2 |\xi|\theta + \tfrac{1+\lambda^2}2 \xi^\parallel + \lambda \xi^\perp)\\
\scriptE f(t,x) \rightsquigarrow \lambda^{-\frac{d+1}q}\scriptE f(\tfrac{1+1/\lambda^2}2 t + \tfrac{1-1/\lambda^2}2 \theta\cdot x, \tfrac{1-1/\lambda^2}2 t\theta + \tfrac{1+1/\lambda^2}2 x^\parallel + \lambda x^\perp),
\end{gather*}
for $\theta \in \mathbb S^{d-1}$,  $\lambda>0$, the parallel and perpendicular parts taken with respect to $\theta$;  
\item Spacetime translations:  
$$f(\xi) \rightsquigarrow e^{-i(t_0,x_0)\cdot(|\xi|, \xi)}f(\xi), \qquad \scriptE f(t,x) \rightsquigarrow \scriptE f(t-t_0, x-x_0),
$$
for $(t_0,x_0) \in \R^{1+d}$.  
\end{CI}

We let $\scriptS$ denote the group whose elements are obtained as compositions of these symmetries, together with multiplication by unimodular complex constants.  

\subsection{Frequency localization}

In this section, we will prove that, after passing to a subsequence and applying symmetries of the operator, a maximizing sequence for \eqref{eq:cone_restr} has a subsequence with good frequency localization.  

It will be useful to introduce some additional terminology.  

We let $\chi_k$ denote a smooth cutoff function supported on  $\{|\xi| \simeq 2^{-k}\}$, with $\sum_k \chi_k \equiv 1$ in measure.  We use $A_k$ to denote the annulus $\{|\xi| \simeq 2^{-k}\}$.   A \textit{sector of angular width $2^{-j}$ at frequency scale $2^{-k}$} is a set of the form
$$
\sigma = \{\xi \in A_k : |\tfrac{\xi}{|\xi|} - \omega_0| < 2^{-j}\},
$$
for some $\omega_0 \in \mathbb S^{d-1}$. If $\sigma,\sigma'$ are two sectors of angular width $2^{-j}$ at the same frequency scale, we say $\sigma \sim \sigma'$ if $\sigma$ and $\sigma'$ are both contained in some common sector $\sigma''$ of angular width $2^{-j+C_1}$, but are not contained in a common sector of angular width $2^{-j+C_0}$, for some $1 < C_0 < C_1$ sufficiently large. For all $j \geq -1$, $k \in \Z$, let $\scriptD_{j,k}$ be a finitely overlapping cover of $A_k$ by sectors of angular width $2^{-j}$. We can construct these covers inductively for each $k$, starting with $\scriptD_{-1, k} = \{A_k\}$. For each $j$, we ensure that each $\sigma \in \scriptD_{j,k}$ is contained in some $\sigma' \in \scriptD_{j-1,k}$. With this definition, we see that 
\begin{equation}\label{E:dyadic whitney decomposition}
    \sum_k \sum_j \sum_{\sigma \sim \sigma' \in \scriptD_{j,k}} \chi_\sigma(\xi) \chi_{\sigma'}(\eta) \simeq 1
\end{equation}
for a.e.\ $\xi \neq \eta \in \R^d \setminus \{0\}$, forming a Whitney decomposition of $(\R^d \setminus \{0\})^2$ minus the diagonal $\{(\xi,\eta):\xi=\eta\}$.  We will denote 
$$
\tau_\sigma := \{(|\xi|,\xi) : \xi \in \sigma\},
$$
the lift of $\sigma$ to the cone.   

  To simplify equations, we will frequently use $\|\cdot\|_p$ to denote $\|\cdot\|_{L^p(\frac{\textup d \xi}{|\xi|})}$.  When the measure is simply $\textup d\xi$, we will indicate it by $\|\cdot\|_{L^p}$.

\begin{lemma}[Bilinear extension between annuli] \label{L:bilinear annuli}
For each $1 < p < p_0$, there exists a permissible $c_0>0$ such that
    \begin{equation} \label{E:bilinear annuli}
        \|\scriptE (f\chi_{k_1})\,\scriptE (f\chi_{k_2})\|_{q/2} \lesssim  2^{-c_0|k_1-k_2|}\|f\chi_{k_1}\|_p\|f\chi_{k_2}\|_p,  
    \end{equation}
for any $f \in L^p(\frac{\textup d \xi}{|\xi|})$.  
\end{lemma}
\begin{proof}
    Since the inequality is symmetric in $k_1$ and $k_2$, we may assume without loss of generality that $k_1 > k_2$. Therefore, we only need to prove the inequality with $(k_1-k_2)$ in the place of $|k_1 - k_2|$.
    
    The Strichartz inequality for the wave equation (\cite[Theorem 1]{Pecher84}) states that if $\tfrac{1}{r} + \tfrac{d-1}{2s} = \tfrac{d-1}{4}$, $s,r \geq 2$, $s \neq \infty$, and $\gamma = \tfrac{d}{2}-\tfrac{1}{r}-\tfrac{d}{s}$, then 
    \[
      \|\scriptE f\|_{L^r(\R; L^s(\R^d))} \lesssim \||\xi|^{\gamma - 1} f\|_{L^2} = \||\xi|^{\gamma - \frac{1}{2}} f\|_2. 
    \]
    The case $p=2$ of \eqref{eq:cone_restr} corresponds to the Strichartz inequality with $(r,s,\gamma)=(2\tfrac{d+1}{d-1},2\tfrac{d+1}{d-1},\tfrac12)$.  We may choose two triples  $(r_j,s_j,\gamma_j)$, $j=1,2$, obeying the preceding conditions, as well as $\tfrac{d-1}{d+1} = \tfrac{1}{r_1} + \tfrac{1}{r_2} = \tfrac{1}{s_1} + \tfrac{1}{s_2}$, $\gamma_1 = \tfrac{1}{2} - c$, and $\gamma_2 = \tfrac{1}{2} + c$ for some $c > 0$. Using H\"older and the annular supports of $f\chi_{k_1}$ and $f\chi_{k_2}$,
    \begin{align*}
        \|\scriptE (f\chi_{k_1})\,\scriptE (f\chi_{k_2})\|_\frac{d+1}{d-1} &\leq \|\scriptE (f\chi_{k_1})\|_{L^{r_1}L^{s_1}} \|\scriptE (f\chi_{k_2})\|_{L^{r_2}L^{s_2}} \\
        &\lesssim \||\xi|^{\gamma_1 - \frac{1}{2}} (f\chi_{k_1})\|_2 \||\xi|^{\gamma_2 - \frac{1}{2}} (f\chi_{k_2})\|_2 \\
        &\simeq 2^{-c(k_1 - k_2)} \|f\chi_{k_1}\|_2 \|f\chi_{k_2}\|_2.
    \end{align*}
    
    Next, we choose $1 \leq p_1 \leq p_0$ so that $p$ lies between $2$ and $p_1$. By Cauchy--Schwarz and \eqref{eq:cone_restr}, 
    \begin{equation} \label{E:elem bilin extn}
        \|\scriptE (f\chi_{k_1})\,\scriptE (f\chi_{k_1})\|_{q(p_1)/2} \lesssim \|f\chi_{k_1}\|_{p_1} \|f\chi_{k_2}\|_{p_1},
    \end{equation}
and \eqref{E:bilinear annuli} follows by complex interpolation.
\end{proof}

Lemma~\ref{L:bilinear annuli} immediately implies a stronger version of \eqref{eq:cone_restr}, which we now state.  

\begin{lemma}[Annular refinement] \label{L:annular sup}
    For some permissible $0 < \theta_0 < 1$, 
    \begin{equation} \label{E:annular sup}
    \|\scriptE f\|_q \lesssim \sup_{k\in\Z} \|\scriptE (f\chi_k)\|_q^{\theta_0} \|f\|_p^{1-\theta_0}.  
    \end{equation}
\end{lemma}

We note that results analogous to Lemma~\ref{L:annular sup} have appeared elsewhere, e.g., \cite{KSV}.  For the convenience of the reader, we give full details, proving a more general (in view of Lemma~\ref{L:bilinear annuli}) lemma below.

\begin{lemma}
    Let $(X, \mu)$ and $(Y, \nu)$ be measure spaces, $1 < p < q < \infty$, and $T: L^p(X) \rightarrow L^q(Y)$ a bounded linear map. Let $\{P_j\}_{j \in \Z}$ be a sequence of bounded linear operators on $L^p(X)$ such that $\sum_j P_j$ converges to the identity in the strong operator topology and $\sum_j \|P_j f\|_p^p \lesssim \|f\|_p^p$ for all $f \in L^p(X)$. Assume that, for some $c_0 > 0$,
    \begin{equation} \label{E:bilinear T}
        \|(T P_j f) (T P_k g)\|_{q/2} \lesssim 2^{-c_0|j-k|}\|f\|_p \|g\|_p
    \end{equation}
    for all $j,k \in \Z$, and $f,g \in L^p(X)$. 
    
    Then there exists $\theta_0 \in (0,1)$ such that, for any $f \in L^p(X)$, 
    \[
        \|T f\|_q \lesssim \sup_j \|T P_j f\|_q^{\theta_0} \|f\|_p^{1-\theta_0}.
    \]
\end{lemma}
\begin{proof}
    If $f = 0$, there is nothing to prove, so without loss of generality, we may assume $\|f\|_p = 1$. For convenience, define $T_j := TP_j$. Let $N := \lceil q\rceil +1$ and let $C > 0$ be a large constant to be chosen later. Since $\tfrac{q}{N}<1$,
    \begin{multline*}
        \|Tf\|_q^q = \int \bigl| \bigl(\sum_j T_j f\bigr)^N \bigr|^{q/N} = \int \bigl| \sum_{j_1,\dots,j_N} \prod_{k=1}^N T_{j_k} f \bigr|^{q/N}\\
        \leq \int \sum_{j_1,\dots,j_N} \bigl| \prod_{k=1}^N T_{j_k} f\bigr|^{q/N} \lesssim \int \sum_{j_1 \leq \dots \leq j_N} \bigl| \prod_{k=1}^N T_{j_k} f\bigr|^{q/N}.
    \end{multline*}
    We split this sum into the terms where $|j_1 - j_N| < C$ and those where $|j_1 - j_N| \geq C$. The first sum is bounded by a constant multiple of $\sum_j \int |T_j f|^q$ by the arithmetic-geometric mean inequality. For the second sum, we apply H\"older's inequality, \eqref{E:bilinear T}, and the arithmetic-geometric mean inequality to obtain
    \begin{align*}
        &\sum_{|j_1 - j_N| \geq C} \int \bigl| \prod_{k=1}^N T_{j_k} f\bigr|^{q/N} 
        \leq \sum_{|j_1-j_N|\geq C} \prod_{k=2}^{N-1} \|T_{j_k} f\|_q^{q/N} \| T_{j_1} f \, T_{j_N} f\|_{q/2}^{q/N} \\
        &\qquad\qquad \lesssim \sum_{\ell = C}^\infty \sum_{j_1 \leq \dots \leq j_N = j_1 + \ell} 2^{-c_0\frac{q}{N} \ell} \|f\|_p^{2q/N} \prod_{k=2}^{N-1} \|T_{j_k} f\|_q^{q/N} \\
        &\qquad\qquad\lesssim \sum_{\ell = C}^\infty  2^{-c_0\frac qN \ell} \sum_{j_1 \leq \dots \leq j_N = j_1 + \ell} \sum_{k=2}^{N-1} \|T_{j_k} f\|_q^{q \frac{N-2}{N}} \\
        &\qquad\qquad\leq \sum_{\ell = C}^\infty  2^{-c_0\frac qN \ell} \sum_{j_1} \ell^{N-2} \sum_{k=j_1}^{j_1+\ell} \|T_k f\|_q^{q \frac{N-2}{N}} \\
        &\qquad\qquad\lesssim \sum_{\ell = C}^\infty  2^{-c_0\frac qN \ell} \ell^{N-2} \sum_j \|T_j f\|_q^{q \frac{N-2}{N}} 
        \simeq \sum_j \|T_j f\|_q^{q \frac{N-2}{N}}.
    \end{align*}
   On the other hand, since $q>p$, we may choose $N$ sufficiently large that $q\frac{N-2}N>p$. Therefore,
    \begin{align*}
        \|Tf\|_q^q &\lesssim (\sup_j \|T_j f\|_q^{q-p} + \sup_j \|T_j f\|_q^{q (\frac{N-2}{N}) - p}) \sum_j \|T_j f\|_q^p \\
        &\lesssim (\sup_j \|T_j f\|_q^{q (\frac{2}{N})} + 1)(\sup_j \|T_j f\|_q^{q (\frac{N-2}{N})-p}) \sum_j \|P_j f\|_p^p \\
        &\lesssim (C\|f\|_p^{q (\frac{2}{N})} + 1)(\sup_j \|T_j f\|_q^{q (\frac{N-2}{N})-p}) \|f\|_p^p \\
        &\lesssim \sup_j \|T_j f\|_q^{q (\frac{N-2}{N})-p}.
    \end{align*}
    Since $\theta_0 := \frac{q (\frac{N-2}{N}) - p}{q} > 0$ for sufficiently large $N$, we have $\|Tf\|_q \lesssim \sup_j \|T_j f\|_q^{\theta_0}$. The result follows from the normalization $\|f\|_p = 1$.
\end{proof}

{
\begin{lemma}[Bilinear extension between sectors]\label{L:bilinear sector refinement}
For $1 < p < p_0$, there exists a permissible $1<s< p$ such that the following holds.  Let $\sigma \sim \tilde\sigma$ be two sectors of angular width $2^{-j}$ at frequency scale $2^{-k}$, and let $f,\tilde f \in L^s(\R^d)$, with supports contained in $\sigma,\tilde\sigma$, respectively.  Then
\begin{equation} \label{E:bilinear sectors}
\|\scriptE f \scriptE \tilde f\|_{q(p)/2} \lesssim 2^{2(k+j)(d-1)(\frac1s-\frac1p)}\|f\|_{ s} \|\tilde f\|_{ s}.
\end{equation}
\end{lemma}}

We note that the condition $1 < p < p_0$ is equivalent to $q(p_0) < q(p) < \infty$.   
\begin{proof}
This is a well-known consequence of the results of \cite{TaoCone, Wo01}.  Indeed, in the case $q = 2\tfrac{d+3}{d+1}$, $k=0$, and $j=C$ (some large permissible constant), this is Theorem~1.1 of \cite{TaoCone} (with $s=2$).  We can remove the restriction $k=0$ by a dilation and the restriction $j=C$ by applying a sectorial expansion \cite[Proposition~2.6]{TaoVargasVega}).

For other values of $q$, we may interpolate with the elementary bilinear extension inequality, \eqref{E:elem bilin extn}, for some $1 \leq p_1 \leq p_0$, chosen so that $q(p)$ lies between $2\tfrac{d+3}{d+1}$ and $q(p_1)$.    
This completes the proof sketch of the lemma.
\end{proof}

Next, we use the scale and sector refinements to bound the norm of $\mathcal{E}f$ using ``chips". Let $\sigma \in \scriptD_{j,k}$ and $\ell \geq 0$. We define
$$
    f_{\sigma,\ell}(\xi) := f\chi_\sigma \chi_{\{|f|<2^\ell \mu(\tau_\sigma)^{-1/p}\|f\|_p\}} 
$$
and 
$$
    f_\sigma^\ell := f_{\sigma,\ell} - f_{\sigma,\ell - 1}.
$$
We further let $f_\sigma^0 = f_{\sigma,0}$.  

To help motivate this definition, we observe that for every $1 \leq s \leq \infty$, 
$$
\|f_{\sigma,\ell}\|_s \leq 2^\ell \mu(\tau_\sigma)^{1/s-1/p}\|f\|_p,
$$
where, recall, $\tau_\sigma$ is the lift of $\sigma$ and $\mu$ is the lift of the measure $\tfrac{\textup d\xi}{|\xi|}$ to the cone.  

\begin{lemma}[Chip refinement] \label{L:pos op}
    There exist $0 < c_1, \theta_0 < 1$ such that
    \begin{equation} \label{E:pos op}
        \|\scriptE f\|_q \lesssim \sup_{k \in \Z} \sup_{j \in \N} \sup_{\sigma \in \scriptD_{j,k}} \sup_{\ell \geq 0} 2^{-c_1 \ell} \|f_\sigma^\ell\|_p^{\theta_0} \|f\|_p^{1-\theta_0}
    \end{equation}
    for all $f \in L^p$.
\end{lemma}

\begin{proof}
    Multiplying by a constant if needed, it suffices to consider the case $\|f\|_p = 1$.  For the moment, let us also suppose that $f = f\chi_{A_0}$, that is, $\supp f \subset \{\xi : 1<|\xi|\leq 2\}$.  Then by \eqref{E:dyadic whitney decomposition}, we can decompose
    \[
        \|\scriptE f\|_q \lesssim \bigl( \int |\sum_j \sum_{\sigma \sim \sigma' \in \scriptD_{j,0}} (\scriptE f_\sigma)(\scriptE f_{\sigma'})|^{q/2}\bigr)^{1/q}.
    \]
    Each product $(\scriptE f_\sigma)(\scriptE f_{\sigma'})$ has Fourier support contained in $\tau_\sigma + \tau_{\sigma'}$. We recall that there exists a boundedly overlapping family of parallelepipeds each containing some sumset $\tau_\sigma + \tau_{\sigma'}$, with $\sigma \sim \sigma' \in \bigcup_j \scriptD_{j,0}$.  Indeed,  when $j$ is fixed, the sums $\sigma+\sigma'$ of related sectors are easily seen to be well-approximated by boundedly overlapping rectangles that have slightly larger widths and the same orientations, so the sums of the lifts are also nearly disjoint.  When $j$ is allowed to vary, we require an additional separation in the first coordinate, which arises because if $\xi$ and $\xi'$ lie in related sectors in $\scriptD_{j,0}$, the angle between them is approximately $2^{-j}$. Therefore, 
    \[
        |\xi| + |\xi'| - |\xi+\xi'| = \frac{2(|\xi||\xi'| - \xi\cdot\xi')}{|\xi| + |\xi'| + |\xi + \xi'|} = \frac{2|\xi||\xi'|(1 - \cos \angle (\xi, \xi'))}{|\xi| + |\xi'| + |\xi+\xi'|} \simeq 2^{-2j}.
    \]

     Let  $t := \min \{\tfrac{q}{2}, (\tfrac{q}{2})'\}$; thus $2t>p>s$, with $s$ as in \eqref{E:bilinear sectors}.  Now we invoke almost orthogonality  (\cite[Lemma 6.1]{TaoVargasVega}) and Lemma \ref{L:bilinear sector refinement} with $k=0$, noting that $\frac1s-\frac1p=\frac1{p'}-\frac1{s'}=\frac{d+1}{(d-1)q}-\frac1{s'}$:  
    \begin{align*}
        &\|\scriptE f\|_q^2 \lesssim \|\sum_j \sum_{\sigma \sim \sigma' \in \scriptD_{j,0}} (\scriptE f_\sigma)(\scriptE f_{\sigma'})\|_{q/2} \lesssim \bigl( \sum_j \sum_{\sigma \sim \sigma' \in \scriptD_{j,0}} \|(\scriptE f_\sigma)(\scriptE f_{\sigma'})\|_{q/2}^t\bigr)^{1/t} \\
        &\qquad \lesssim \bigl( \sum_j \sum_{\sigma \sim \sigma' \in \scriptD_{j,0}} 2^{2jt(\frac{d+1}{q}-\frac{d-1}{s'})} \|f_\sigma\|_s^t \|f_{\sigma'}\|_s^t\bigr)^{1/t} \\
        &\qquad \lesssim \bigl( \sum_j \sum_{\sigma'' \in \scriptD_{j-1,0}} 2^{2jt(\frac{d+1}{q} - \frac{d-1}{s'})} \|f_{\sigma''}\|_s^{2t} \bigr)^{1/t},
    \end{align*}
    where $\sigma \cup \sigma' \subseteq \sigma''$. 
    Let $\max \{\tfrac{p}{2t}, \tfrac{s}{p}\} < \theta < 1$. After reindexing $\sigma'' \mapsto \sigma$ and noting that $|\sigma| \sim 2^{j(d-1)}$, for $\sigma \in \scriptD_{j-1,0}$, then applying H\"older's inequality and the fact that $2t>p>s$, we see that for any $c_1>0$, 
    \begin{align*}
        &\|\scriptE f\|_q^{2t} 
        \lesssim \sum_j\sum_{\sigma \in \scriptD_{j-1,0}} |\sigma|^{2t(\frac{1}{s'} - \frac{1}{p'})} (\sum_{\ell \geq 0} \|f_\sigma^\ell\|_s^s)^{2t/s}  \\
        &\quad\leq \bigl(\sup_j \sup_{\sigma \in \scriptD_{j-1,0}} \sup_{\ell \geq 0} 2^{-\frac{c_1 \ell}{1-\theta}} |\sigma|^{2(\frac{1}{s'} - \frac{1}{p'})} \|f_\sigma^\ell\|_s^2\bigr)^{(1-\theta)t} 
        \\
        &\quad\qquad \times  
        \sum_j \sum_{\sigma \in \scriptD_{j-1,0}} |\sigma|^{2t\theta (\frac{1}{s'} - \frac{1}{p'})} (\sum_{\ell \geq 0} 2^{\frac{c_1 \ell s}{2}} \|f_\sigma^\ell\|_s^{\theta s})^{2t/s}  \\
        &\quad\lesssim \bigl(\sup_j \sup_{\sigma \in \scriptD_{j-1,0}} \sup_{\ell \geq 0} 2^{-c_1 \ell} \|f_\sigma^\ell\|_p^{2t(1-\theta)}\bigr) 
        \bigl(\sum_j \sum_{\sigma \in \scriptD_{j-1,0}} |\sigma|^{2t\theta(\frac{1}{p} - \frac{1}{s})} \sum_{\ell \geq 0} 2^{t c_1 \ell} \|f_\sigma^\ell\|_s^{2t\theta}\bigr).
    \end{align*}
    (We recall that for $\sigma \subseteq A_0$, $|\sigma| \simeq \mu(\tau_\sigma)$.) 
    Because we can take $c_1>0$ arbitrarily small, it remains to prove that
\begin{equation}\label{E:refined inequality eq 1}
        \sum_j \sum_{\sigma \in \scriptD_{j-1,0}} |\sigma|^{2t\theta(\frac{1}{p} - \frac{1}{s})} \|f_\sigma^\ell\|_s^{2t\theta}
    \end{equation}
    decays geometrically in $\ell$. 
    For $\ell = 0$, we apply H\"older and $2t\theta > p$ to obtain
    \begin{align*}
        \eqref{E:refined inequality eq 1} &\leq \sum_j \sum_{\sigma \in \scriptD_{j-1,0}} |\sigma|^{\frac{2t\theta}{p} - 1} \|f_\sigma^0\|_{2t\theta}^{2t\theta} \lesssim \sum_j 2^{-j(d-1)(\frac{2t\theta}{p} - 1)} \int_{|f| < 2^\frac{j(d-1)}{p}} |f|^{2t\theta} \\
        &\lesssim \int \sum_{j \gtrsim \frac{p}{d-1} \log |f|} 2^{-j(d-1)(\frac{2t\theta}{p} - 1)} |f|^{2t\theta} 
        \simeq \int |f|^p \simeq 1.
    \end{align*}
    For $\ell \geq 1$, we apply H\"older twice to obtain
    \begin{align*}
        \eqref{E:refined inequality eq 1} &\leq \Big( \sum_j \sum_{\sigma \in \scriptD_{j-1,0}} 2^{- \frac{j(d-1)(s-p)}{p}} \|f_\sigma^\ell \|_s^s\Big)^\frac{2t\theta}{s} 
        \!\!= \Big( \sum_j 2^{\frac{j(d-1)(p-s)}{p}} \!\!\int_{|f| 
        \simeq 2^{\frac{j(d-1)}{p} + \ell}} |f|^s \Big)^\frac{2t\theta}{s} \\
        &\lesssim \Big( \sum_j 2^{-\ell(p-s)} \int_{|f| \simeq 2^{\frac{j(d-1)}{p} + \ell}} |f|^p \Big)^\frac{2t\theta}{s} 
        \leq 2^{-\ell(p-s)} \|f\|_p^\frac{2t\theta p}{s} \simeq 2^{-\ell(p-s)}.
    \end{align*}
    
    Now consider a function $f$ with arbitrary support. By Lemma \ref{L:annular sup} and scaling, there exists $\theta_0 \in (0,1)$ such that
    \[
        \|\scriptE f\|_q \lesssim \sup_k \|\scriptE (f\chi_k)\|_q^{\theta_0} \|f\|_p^{1-\theta_0} = \sup_k \|\scriptE g_k\|_q^{\theta_0} \|f\|_p^{1-\theta_0},
    \]
    where $g_k(\xi) = 2^{k\frac{d-1}{p}}f(2^k \xi) \chi_1(\xi)$. For all $\ell$ and $\sigma \in \scriptD_{j,0}$, we see that
    \begin{multline*}
        (g_k)_\sigma^\ell(\xi) = g_k(\xi) \chi_\sigma(\xi) \chi_{\{|g_k| \simeq 2^\ell \mu(\tau_\sigma)^{-1/p}\|g_k\|_p\}}(\xi) \\
        = 2^{k\frac{d-1}{p}} f(2^k \xi) \chi_{2^k \sigma}(2^k \xi) \chi_{\{|f(2^k \cdot)| \simeq 2^\ell \mu(\tau_{2^k \sigma})^{-1/p} \|f\|_p\}}(\xi) = 2^{k\frac{d-1}{p}}f_{2^k\sigma}^\ell(2^k \xi).
    \end{multline*}
    By construction, $2^k \sigma \in \scriptD_{j,k}$. Since we have already proved the result for functions supported on an annulus and $\supp g_k \subset \{1 < |\xi| \leq 2\}$,
    \begin{align*}
        \sup_k \|\scriptE g_k\|_q^{\theta_0} \|f\|_p^{1-\theta_0} &\lesssim \sup_k \sup_j \sup_{\sigma \in \scriptD_{j,0}} \sup_\ell 2^{-c_0 \theta_0 \ell} \|\scriptE (g_k)_\sigma^\ell\|_q^{\theta   \theta_0} \|g_k\|_p^{\theta_0(1-\theta)} \|f\|_p^{1-\theta_0} \\
        &= \sup_{j,k} \sup_{\sigma \in \scriptD_{j,k}} \sup_\ell 2^{-c_0\theta_0 \ell} \|\scriptE f_\sigma^\ell\|_q^{\theta   \theta_0} \|f\|_p^{1-\theta   \theta_0}.
    \end{align*}
    This completes the proof of the lemma.
\end{proof}

\begin{lemma}[Chip extraction] \label{L:chips}
For $1 < p < p_0$, there exists a permissible sequence $\rho_m  \searrow 0$, such that for every $f \in L^p(\frac{\textup d \xi}{|\xi|})$, there exists a sequence of sectors $\sigma_m$, such that if $r^m$ and $f^m$ are defined recursively by
$$
r^0:=f, \quad r^m:=r^{m-1}-f^m, \quad f^m:=r^{m-1}\chi_{\sigma_m}\chi_{\{|f|<2^m\mu(\tau_{\sigma_m})^{-1/p}\|f\|_p\}},
$$
then 
$$
\|\scriptE h^m\|_{q} \leq \rho_m \|f\|_p,
$$
for every measurable $h^m$ with $|h^m| = |r^m|\chi_E$, for some measurable set $E$.  
\end{lemma}

\begin{proof}
We will prove that the lemma holds with $\rho_m = Cm^{-\theta_0/p}$, with $\theta_0$ taken from \eqref{E:pos op}.  

We may assume that $\|f\|_p=1$.  By the dominated convergence theorem, given $r^{m-1}$, we may choose $\sigma_m$ to maximize $\|f^m\|_p$.  If $|h^m| = |r^m|\chi_E$ has $\|\scriptE h^m\|_q > \rho_m$, then (since $\|h^m\|_p \leq 1$),
$$
\sup_{\ell\geq 0}\sup_{\sigma } 2^{-c_0\ell}\|(h^m)_\sigma^\ell\|_p^\theta > \rho_m.
$$
Then there exists some $\ell \leq -\tfrac1{c_0} \log_2 \rho_m$ such that $\|(h^m)_\sigma^\ell\|_p > \rho_m^{1/\theta_0}$.  Since 
$$
|(h^m)_\sigma^\ell| = |r^m|\chi_E\chi_\sigma\chi_{\{|r^m| < 2^\ell \mu(\tau_\sigma)^{-1/p}\|h^m\|_p\}},
$$
and $\|h^m\|_p \leq 1$, 
$$
\|r^m\chi_\sigma\chi_{\{|r^m|<2^\ell\mu(\tau_\sigma)^{-1/p}\}}\|_p > \rho_m^{1/\theta}.
$$
Since $\ell \leq -\tfrac1{c_0}\log_2\tfrac{C}{m^{1/p}} \leq m/2$ (for $C$ sufficiently large), by the maximal condition on $\sigma_{m+1}$, $\|f^{m+1}\|_p > \rho_m^{1/\theta}$.  In fact, for any $m/2 \leq m' \leq m$, exactly the same arguments shows $\|f^{m'}\|_p > \rho_m^{1/\theta}$.  By construction, the $f^{m'}$ have disjoint supports and $\sum|f^{m'}| \leq |f|$.  Therefore $\|f\|_p > m/2 \rho_m^{p/\theta} \geq 1$, a contradiction.  Tracing back and applying Lemma~\ref{L:pos op}, we must have had $\|\scriptE h^m\|_{q} \leq \rho_n$ all along.  
\end{proof}

\begin{proposition}[One big chunk] \label{P:one chunk}
Let $1 < p < p_0$ and $\{f_n\}\subseteq L^p(\frac{\textup d \xi}{|\xi|})$, with 
$$
\|f_n\|_p \equiv 1, \qtq{and} \lim_{n\to\infty}\|\scriptE f_n\|_{q} = \mathbf{A}_{p,q}.
$$
There exist symmetries $S_n \in \scriptS$ such that, setting $\tilde f_n:=S_nf_n$, 
\begin{equation} \label{E:one chunk}
    \lim_{R \to \infty} \lim_{n\to\infty} \|\tilde f_n(1-\chi_{\{R^{-1} \leq |\xi| \leq R\}}\chi_{\{|\tilde f_n|\leq R\}})\|_p = 0.
\end{equation}
\end{proposition}

\begin{proof}
For each $n$, we let $\{\sigma_n^m\}$, $\{f_n^m\}$, $\{r_n^m\}$ denote the sequences of sectors and functions (resp.) associated to $f_n$, as defined in the proof of Lemma~\ref{L:chips}.  Roughly, we will show that the $\{f_n^m\}$ are all either negligible (as $n \to \infty$) or localized to sectors at comparable scales and angular widths.  

We begin by identifying an appropriate rescaling.  By the triangle inequality and Lemma~\ref{L:chips}, there exists a permissible $M$ such that for all sufficiently large $n$, there exists $m_n \leq M$ such that $\|f_n^{m_n}\|_p \gtrsim 1$.  Applying symmetries $S_n\in\mathcal S$, we may assume that each $\sigma_n^{m_n}$ has angular width $\gtrsim 1$ and lives at frequency scale 1.  By separately considering a permissible number of subsequences, we may assume that $m_n = m_0$ for all $n$.  

Let $\{2^{-k_n^m}\}$ and $\{2^{-j_n^m}\}$ denote the frequency scale and angular width (resp.) of $\sigma_n^m$.  We say that $m \in \N$ is `good' if every subsequence $\{n_\ell\}$ possesses a further subsequence $\{\tilde n_\ell\}$ such that $\lim_{\ell \to \infty}\|f_{\tilde n_\ell}^m\|_p = 0$, or $k_{\tilde n_\ell}^m$ and $j_{\tilde n_\ell}^m$ are constant (in $\ell$).  We will, of course, say that a not-good $m$ is `bad.'  Naturally, $m_0$ is good, and, in fact, we will show that there are no bad $m$'s.  

Suppose that the index $m_1$ is bad.  Then, by Cantor's diagonalization argument, there exists a subsequence $\{n_\ell\}$ along which
$$
\lim_{\ell \to \infty} \|f_{n_\ell}^m\|_p, \: \lim_{\ell \to \infty} k_{n_\ell}^m, \qtq{and} \lim_{\ell \to \infty} j_{n_\ell}^m
$$
all exist (with the latter two possibly infinite) for all $m$, such that 
$$
\lim_{\ell \to \infty} \|f_{n_\ell}^{m_1}\|_p > 0; \qtq{and} \lim_{\ell \to \infty} k_{n_\ell}^{m_1}=\infty \qtq{or} \lim_{\ell \to \infty} j_{n_\ell}^{m_1}=\infty.  
$$
We will derive a contradiction by showing that $f_{n_\ell}$ is not maximizing.  To simplify expressions, we will drop the extra subscript $\ell$.  

For convenience we modify the definitions of good and bad slightly so that (along our subsequence) 
$$
\lim_{n \to \infty} \|f_n^m\|_p=0; \qtq{or} \lim_{n \to \infty} k_{n}^m \in \Z \,\qtq{and} \lim_{n \to \infty} j_{n}^m \in \N
$$
for each good $m$, and 
$$
\lim_{n \to \infty} \|f_n^m\|_p>0; \qtq{and} \lim_{n \to \infty} k_{n}^m \in \{\pm\infty\}\, \qtq{or} \lim_{n \to \infty} j_{n}^m =\infty,
$$
for each bad $m$.  We note that $m_0$ is still good (with $\lim_{n \to \infty} \|f_n^{m_0}\|_p>0$), and $m_1$ is still bad.  

For $M \geq 0$, we define
\begin{gather*}
    G_n^M :=\sum_{\textrm{good}\,m \leq M} f_n^m \qquad\qquad 
    B_n^M :=\sum_{\textrm{bad}\,m \leq M} f_n^m.
\end{gather*}
Therefore, $f_n = G_n^M+B_n^M+r_n^M$ for all $n$, with the summands on the right hand side having disjoint supports.  

We begin by showing that the $r_n^M$ are small in $L^p$.  By Lemma~\ref{L:chips}, 
$$\lim_{m \to \infty} \limsup_{n \rightarrow \infty} \|\scriptE r_n^m\|_{q} = 0.$$
Thus, 
\begin{equation} \label{E:rnm to 0 Lp}
\lim_{m \to \infty} \limsup_{n \rightarrow \infty} \|r_n^m\|_p = 0,
\end{equation}
for if not, 
$$
\frac{\|\scriptE(f_n-r_n^m)\|_{q}}{\|f_n-r_n^m\|_p} \geq \frac{\|\scriptE f_n\|_{q}-\|\scriptE r_n^m\|_{q}}{(\|f_n\|_p^p-\|r_n^m\|_p^p)^{1/p}} 
$$
would exceed $\mathbf{A}_{p,q} = \lim_n \frac{\|\scriptE f_n\|_{q}}{\|f_n\|_p}$, for some choice of $n,m$ sufficiently large, which contradicts our assumption that $\{f_n\}$ is maximizing. 

Next, we will show that 
\begin{equation} \label{E:good bad decoup Lq}
\lim_{n \to \infty} \|\scriptE(G_n^M+B_n^M)\|_{q}^{q} - \|\scriptE G_n^M\|_{q}^{q} - \|\scriptE B_n^M\|_{q}^{q} = 0,
\end{equation}
for all $M$.  It suffices to prove that for every good $m$ and bad $m'$, 
\begin{equation}\label{E:bilin good bad}
    \lim_{n \rightarrow \infty} \|\scriptE f_n^m \,\scriptE f_n^{m'}\|_{q/2} = 0.  
\end{equation}
If $\|f_n^M\|_p \to 0$, then \eqref{E:bilin good bad} follows from Cauchy--Schwarz.  Thus, we may assume that $k_n^m,j_n^m$ are eventually constant, while either $k_n^{m'}$ or $j_n^{m'}$ has an infinite limit.  In the former case, \eqref{E:bilin good bad} follows directly from the annular decoupling \eqref{E:bilinear annuli}, so we may assume that $k_n^{m'}$ is eventually constant and $j_n^{m'} \to \infty$.  We set $\frac1{p_1}:=\frac2p-\frac1{p_0}$, using a slightly smaller value of $p_0>p$ if $p$ is very close to 1.  We then have $1 <p_1<p<p_0$ and $\frac2{q(p)}=\frac1{q(p_0)}+\frac1{q(p_1)}$. (We recall the definition \eqref{eq:RCrange} of $q(p)$.)  By H\"older's inequality and the definition of the $f_n^m$, $\|f_n^m\|_{p_0}$ is bounded and 
$$
\|f_n^{m'}\|_{p_1} \lesssim \mu(\tau_{\sigma_n^{m'}})^{1/p_1-1/p}\|f_n^{m'}\|_p \lesssim 2^{-dj_n^{m'}(1/p_1-1/p)} \to 0.
$$
Another application of H\"older gives \eqref{E:bilin good bad}, and therefore \eqref{E:good bad decoup Lq}.  

Finally, by \eqref{E:rnm to 0 Lp} and \eqref{E:good bad decoup Lq},
\begin{align*}
    {\bf A}_{p,q}^{q} &= \lim_{M \to \infty} \lim_{n \to \infty} \|\scriptE G_n^M\|_{q}^{q} + \|\scriptE B_n^M\|_{q}^{q}\\
    &\leq {\bf A}_{p,q}^{q} \lim_{M \to \infty} \lim_{n \to \infty} \|G_n^M\|_p^{q} + \|B_n^M\|_p^{q}\\
    &={\bf A}_{p,q}^{q} \lim_{M \to \infty} \lim_{n \to \infty} \max\{\|G_n^M\|_p,\|B_n^M\|_p\}^{q-p} \leq {\bf A}_{p,q}^{q},
\end{align*}
where the last inequality follows from  $\|G_n^M\|_p^p + \|B_n^M\|_p^p \leq 1$.  Comparing the right and left sides, we see that all inequalities must be equalities.  However,  
\[
\lim_{M \to \infty} \lim_{n \to \infty} \|G_n^M\|_p^p \leq 1-\lim_{n \to \infty}\|f_n^{m_1}\|_p^p , \qquad \lim_{M \to \infty} \lim_{n \to \infty} \|B_n^M\|_p^p \leq 1-\lim_{n \to \infty}\|f_n^{m_0}\|_p^p,
\]
which are both less than 1, a contradiction.  Tracing back, the only possibility is that $m_1$ was not bad. 
Finally, we prove \eqref{E:one chunk} (with $\tilde f_n = f_n$).  By the triangle inequality and 
\[
\lim_{M \to \infty} \lim_{R \to \infty} \lim_{n \to \infty} \|r_n^M(1-\chi_{\{R^{-1} \leq |\xi| \leq R\}}\chi_{\{|f_n| \leq R\}})\|_p \leq \lim_{M \to \infty} \lim_{n \to \infty} \|r_n^M\|_p = 0,
\]
we have
\begin{align*}
&\lim_{R \to \infty} \lim_{n \to \infty} \|f_n(1-\chi_{\{R^{-1} \leq |\xi| \leq R\}}\chi_{\{|f_n| \leq R\}})\|_p \\
&\qquad = 
\lim_{M \to \infty} \lim_{R \to \infty} \lim_{n \to \infty} \|G_n^M(1-\chi_{\{R^{-1} \leq |\xi| \leq R\}}\chi_{\{|f_n| \leq R\}})\|_p.
\end{align*}
Considering a single good $f_n^m$,
\[
\lim_{R \to \infty} \lim_{n \to \infty} \|f_n^m(1-\chi_{\{R^{-1} \leq |\xi| \leq R\}}\chi_{\{|f_n| \leq R\}})\|_p = 0,
\]
because every subsequence (in $n$) has a further subsequence along which either $\|f_n^m\|_p \to 0$ or the parameters associated to the $\sigma_n^m$ remain bounded.  The proposition now follows from the triangle inequality.  
\end{proof}

\subsection{Spatial localization}

To obtain the spatial localization, we will apply the following simple consequence of the profile decomposition for the wave equation, which may be found in \cite{BG, Bulut, KM, Ra12}.

\begin{lemma}[Frequency localized $L^2$ profile decomposition] \label{L:localized profile}
Let $R>0$ and let $\{f_n\}$ be a sequence of measurable functions supported on $\{R^{-1} < |\xi| < R\}$ and obeying $|f_n| < R$.  Then there exist $\{\phi^j\}_{j \in \N} \subseteq L^2$ and $\{(t_n^j,x_n^j)\}_{n,j \in \N} \subseteq \R^{1+d}$ such that, after passing to a subsequence,
\begin{CI}
    \item $\phi^j = \wklim e^{i(t_n^j,x_n^j) \cdot (|\xi|,\xi)} f_n$
    \item $\lim_{n \to \infty} \bigl(|t_n^j-t_n^{j'}| + |x_n^j-x_n^{j'}|\bigr) = \infty$, $j \neq j'$
\end{CI}
and the remainder terms
\begin{equation} \label{E:wnj}
    w_n^J(\xi) := f_n(\xi) - \sum_{j=1}^J e^{-i(t_n^{j}, x_n^{j})\cdot (|\xi|, \xi)} \phi^{j}(\xi)
\end{equation}
obey 
\begin{CI}
\item $\lim_{J \to \infty} \limsup_{n \to \infty} \|\scriptE w_n^J\|_{\frac{2(d+1)}{d-1}} = 0$
\item $\lim_{n \to \infty} \|f_n\|_2^2 - \sum_{j=1}^J \|\phi^j\|_2^2 - \|w_n^J\|_2^2 = 0$, for all $1 \leq J < \infty$
\item $\lim_{n \to \infty} \|\scriptE f_n\|_{\frac{2(d+1)}{d-1}}^{\frac{2(d+1)}{d-1}} - \sum_{j=1}^J\|\scriptE \phi^j\|_{\frac{2(d+1)}{d-1}}^{\frac{2(d+1)}{d-1}} - \|\scriptE w_n^J\|_{\frac{2(d+1)}{d-1}}^{\frac{2(d+1)}{d-1}} = 0$,  $1 \leq J < \infty$.
\end{CI}
\end{lemma}

We note that the hypothesis that $|f_n| \leq R\chi_{\{R^{-1} \leq |\xi| \leq R\}}$ is harmless in view of Proposition~\ref{P:one chunk}.

\begin{proof}
The result is a direct application of the profile decomposition in (e.g.) \cite[Theorem 3.1]{Ra12}.  Indeed, setting
$$
(u_{0,n},u_{1,n}):= (\check{f_n},\tfrac{i}{|\nabla|}\check{f_n}),
$$
we see that $\scriptE f_n = S(t)(u_{0,n},u_{1,n})$, in the notation of the above-mentioned articles, and that, thanks to our frequency localization, $(u_{0,n},u_{1,n})$ is bounded in $\dot{H}^1 \times L^2$. Moreover, in the terminology of \cite{Bulut}, the sequence is ``$1$-oscillatory". Therefore, \cite[Lemma 3.8]{Bulut} yields $(V_0^j, V_1^j)$ such that $\phi^j = \widehat{V_0^j}$ satisfies all of our conclusions. The claim that $e^{i(t_n^j, x_n^j)\cdot(|\xi|, \xi)} f_n \rightharpoonup \phi^j$ for all $j$ follows from the construction of $V_0^j$ in \cite{Bulut}.
\end{proof}

\begin{proposition}[$L^p$ profile decomposition]\label{P:L^p decomp}
    Let $R > 0$. Let $f_n$ be measurable functions such that $\supp f_n \subset \{R^{-1} < |\xi| < R\}$ and $|f_n| < R$ almost everywhere. Then there exist a subsequence in $n$, sequences $(t_n^{j}, x_n^{j}) \in \R^{1+d}$ and bounded functions $\phi^{(j)}$ with $\supp \phi^{j} \subset \{|\xi| < R\}$ such that the following statements hold, with $w_n^J$ as in \eqref{E:wnj}.
    \begin{enumerate}
        \item $\lim_{n \rightarrow \infty} |x_n^{j} - x_n^{j'}| + |t_n^{j} - t_n^{j'}| = \infty$ for all $j \neq j'$;
        \item $\lim_{J \rightarrow \infty} \limsup_{n\rightarrow \infty} \|\scriptE w_n^J\|_q = 0$;
        \item $\liminf_{n \rightarrow \infty} \|f_n\|_p \geq (\sum_{j=1}^\infty \|\phi^{j}\|_p^{\tilde p})^{1/\tilde p} $ for $\tilde p = \max \{p, p'\}$;
        \item $\lim_{n\rightarrow \infty} \|\scriptE f_n\|_q^q - \sum_{j=1}^J \|\scriptE \phi^{j}\|_q^q - \|\scriptE w_n^J\|_q^q = 0$ for all $J$; and
        \item $e^{i(t_n^{j}, x_n^{j})\cdot (|\xi|, \xi)} f_n \rightharpoonup \phi^{j}$ as $n\to\infty$, for all $j$.
    \end{enumerate}
\end{proposition}
\begin{proof}
    Conclusions 1 and 5 follow immediately from Lemma~\ref{L:localized profile}.
    
    We may assume $p \neq 2$, as we are otherwise in the case covered by Lemma~\ref{L:localized profile}. Let $1 < p_1 < \tfrac{2d}{d-1}$ be such that $p$ lies strictly between $2$ and $p_1$. Conclusion 4 follows by the Br\'{e}zis--Lieb lemma and induction since $\scriptE(e^{i(t_n^1,x_n^1)\cdot(|\xi|,\xi)}f_n-\phi^1) \to 0$ pointwise and, for $j>1$, $\scriptE(e^{i(t_n^{j},x_n^{j})\cdot(|\xi|,\xi)}w_n^{j-1}-\phi^{j}) \to 0$ pointwise.  Moreover, conclusion 4 also holds just as well for the exponents $(p_1, q(p_1))$, since $f_n$ is bounded in $L^{p_1}$. By H\"older's inequality, conclusion 2 follows from conclusion 4 and the second conclusion in Lemma~\ref{L:localized profile}. 
    
    We need to work a little harder for conclusion 3. Let $\varepsilon > 0$ and $\varphi, \psi$ be smooth, nonnegative, compactly supported functions such that $\|\varphi\|_\infty = \int \psi = 1$ and $\|\phi^{j} - \psi * (\varphi\phi^{j})\|_p < \varepsilon$ for all $j \leq J$. We assume further that $\supp \psi \subset \{|\xi| < \tfrac{1}{2}R^{-1}\}$ and $\supp \varphi \subset \R^d \setminus \{0\}$. Let
    \[
        \pi_n^j(f) := e^{-i(t_n^{j}, x_n^{j})\cdot (|\xi|, \xi)} \psi * (\varphi e^{i(t_n^{j}, x_n^{j})\cdot (|\cdot|, \cdot)} f).
    \]
    Since $\psi*(\varphi\phi^{j})$ is compactly supported and bounded above by 1, conclusion 5 gives us
    \begin{multline*}
        \bigl(\sum_{j=1}^J \|\phi^{j}\|_p^{\tilde p}\bigr)^{1/\tilde p} \leq \bigl(\sum_{j=1}^J \|e^{-i(t_n^{j}, x_n^{j})\cdot(|\xi|, \xi)} \psi * (\varphi \phi^{j})\|_p^{\tilde p}\bigr)^{1/\tilde p} + o_{\varepsilon,J}(1) \\
        \leq \bigl(\sum_{j=1}^J \|\pi_n^j (f_n)\|_p^{\tilde p}\bigr)^{1/\tilde p}  + C_J \bigl(\sum_{j\neq j' \leq J} \|\pi_n^j e^{-i(t_n^{j'}, x_n^{j'})\cdot (|\xi|, \xi)} \phi^{j'}\|_p^{\tilde p}\bigr)^{1/\tilde p}  \\
        + \|\pi_n^J (w_n^J)\|_p + o_{\varepsilon,J}(1).
    \end{multline*}
    The middle two terms on the right-hand side go to zero as $n \rightarrow \infty$ by the dominated convergence theorem and integration by parts. Sending $\varepsilon \rightarrow 0$, the last term disappears as well, so it remains to estimate the first term as $n \rightarrow \infty$. Define the vector-valued functional
    \[
        \Pi_n^J(f) := (\pi_n^j(f))_{j=1}^J.
    \]
    To prove conclusion 3, it suffices to show that $\limsup_{n\rightarrow \infty} \|\Pi_n^J\|_{L^p \rightarrow \ell^{\tilde p} L^p} \leq 1$. 
    
    Since $\|\psi * (\varphi g)\|_s \leq \|g\|_s$ for all $1 \leq s \leq \infty$ and $g \in L^s$, we know that $\|\Pi_n^J\|_{L^1 \rightarrow \ell^\infty L^1}, \|\Pi_n^J\|_{L^\infty \rightarrow \ell^\infty L^\infty} \leq 1$. Therefore, by complex interpolation and duality, we need to prove that  
    \[
        \limsup_{n\rightarrow \infty} \|(\Pi_n^J)^* \mathbf{g}\|_2^2 \leq \|\mathbf{g}\|_2^2
    \]
    for all $\mathbf{g} = (g_1, \dots, g_J) \in (L^2)^J$. We expand
    \begin{multline*}
        \|(\Pi_n^J)^* \mathbf{g}\|_2^2 = \big\|\sum_j e^{-i(t_n^{j}, x_n^{j})\cdot (|\xi|, \xi)} \varphi (e^{i(t_n^{j}, x_n^{j})\cdot (|\cdot|, \cdot)} g_j * \psi)\big\|_2^2 \\
        \leq \sum_j \|(\pi_n^j)^* g_j\|_2^2 + \sum_{j\neq j'} \int |\pi_n^{j'} (\pi_n^j)^* g_j \overline{g_{j'}}| \leq \|\mathbf{g}\|_2^2 + \sum_{j \neq j'} \|\pi_n^{j'} (\pi_n^j)^* g_j\|_2 \|g_{j'}\|_2.
    \end{multline*}
    Finally, let $h_j := e^{i(t_n^{j}, x_n^{j})\cdot(|\xi|, \xi)} g_j$. Since $\varphi$ and $\psi$ are smooth and compactly supported and $\xi \mapsto (t_0, x_0)\cdot(|\xi|, \xi)$ has no critical points on $\tfrac{1}{2}R^{-1} < |\xi| < R$, stationary phase and conclusion 1 give us
    \begin{multline*}
        \|\pi_n^{j'} (\pi_n^j)^* g_j\|_2 \lesssim \|\psi * (e^{i(t_n^{j'} - t_n^{j}, x_n^{j'} - x_n^{j}) \cdot (|\xi|, \xi)} |\varphi|^2 (h_j * \psi))\|_\infty \\
        \lesssim (1+|(t_n^{j'} - t_n^{j}, x_n^{j'} - x_n^{j})|)^{-d/2} \rightarrow 0,
    \end{multline*}
    which proves conclusion 3.
\end{proof}

\begin{lemma}[One big bubble]\label{L:one big bubble}
    For every $\varepsilon > 0$ there exists $\delta > 0$ such that the following holds. For all sequences $\{f_n\}$ with $\supp f_n \subset \{R^{-1} < |\xi| < R\}$, $|f_n| < R$, $\|f_n\|_p = 1$ for all $n$, and $\lim_{n \rightarrow \infty} \|\scriptE f_n\|_q > (1-\delta){\bf A}_{p,q}$, there exist a subsequence in $n$, a bounded function $\phi$ supported on $\{|\xi| < R\}$, and a sequence $(t_n,x_n) \in \R^{1+d}$ such that
    \[
        \limsup_{n\rightarrow \infty} \|f_n - e^{-i(t_n,x_n)\cdot(|\xi|, \xi)} \phi\|_p < \varepsilon.
    \]
\end{lemma}
\begin{proof}
    Assume that we have chosen $\delta > 0$, whose value we will determine later. By statements 2 and 4 of Proposition \ref{P:L^p decomp},
    \begin{multline*}
        (1-\delta)^q {\bf A}_{p,q}^q \leq \lim_{n\to\infty} \|\scriptE f_n\|_q^q \leq \lim_{J \rightarrow \infty} \left(\sum_{j=1}^J \|\scriptE \phi^{j}\|_q^q + \limsup_{n\rightarrow \infty} \|\scriptE w_n^J\|_q^q\right) \\
        \leq \lim_{J\to\infty} {\bf A}_{p,q}^q \sum_{j=1}^J \|\phi^{j}\|_p^q \leq {\bf A}_{p,q}^q \sup_{j\in\N} \|\phi^{j}\|_p^{q-\tilde p} \lim_{J\to\infty} \sum_j \|\phi^{j}\|_p^{\tilde p} \\
        \leq {\bf A}_{p,q}^q \sup_{j\in\N} \|\phi^{j}\|_p^{q-\tilde p}.
    \end{multline*}
    Therefore, there exists $j_0 \in \N$ such that
    \[
        \|\phi^{(j_0)}\|_p \geq (1-\delta)^\frac{q}{q- \tilde p}.
    \]
    Since $e^{i(t_n^{(j_0)}, x_n^{(j_0)})\cdot(|\xi|, \xi)} f_n \rightharpoonup \phi^{(j_0)}$, as $n\to\infty$, by statement 5 from Proposition \ref{P:L^p decomp} and $q > \tilde p$ by definition, the uniform convexity of $L^p$ implies that we can take $\delta$ small enough to satisfy the claim.
\end{proof}

\begin{proof}[Proof of Theorem~\ref{T:extremizers exist intro}]
    Let $\{f_n\}\subset L^p(\frac{\textup d \xi}{|\xi|})$ be a non-zero sequence such that $\|\scriptE f_n\|_q \rightarrow {\bf A}_{p,q}$ and $\|f_n\|_p = 1$ for all $n$. By Proposition \ref{P:one chunk}, after passing to a subsequence and applying a sequence of symmetries to $f_n$, 
    \[
        \lim_{m \rightarrow \infty} \lim_{n\rightarrow \infty} \|\scriptE f_n^m\|_q/\|f_n^m\|_p = {\bf A}_{p,q},
    \]
    where $f_n^m := f_n \chi_{\{m^{-1} < |\xi| < m\}} \chi_{\{|f_n| < m\}}$ for all $m \in \N$.
    
    Lemma \ref{L:one big bubble} yields a bounded function $\phi^m$ such that after modulating the $f_n$ and passing to another subsequence,
    \[
        \lim_{m \rightarrow \infty} \limsup_{n\rightarrow \infty} \|f_n^m - \phi^m\|_p = 0.
    \]
    By the triangle inequality and Proposition \ref{P:one chunk}, we can drop the truncation of $f_n$ to find that
    \[
        \lim_{m \rightarrow \infty} \limsup_{n\rightarrow \infty} \|f_n - \phi^m\|_p = 0.
    \]
    Therefore, for all $\varepsilon > 0$ there exists $M \in \N$ such that, for every $m,m' > M$, $\|\phi^m - \phi^{m'}\|_p < \varepsilon$. Hence the sequence $\{\phi^m\}$ is Cauchy and thus converges to some $\phi \in L^p$. Since $f_n \rightarrow \phi$ in $L^p$, $\phi$ is a maximizer.
\end{proof}

        \section*{Acknowledgements}
GN and DOS were supported by the EPSRC New Investigator Award ``Sharp Fourier Restriction Theory'', grant no.\@ EP/T001364/1, and FCT/Portugal through project UIDB/04459/2020 with DOI identifier 10-54499/UIDP/04459/2020.
DOS acknowledges partial support from the Deutsche Forschungsgemeinschaft under Germany's Excellence Strategy -- EXC-2047/1 -- 390685813 and is grateful to René Quilodrán for valuable discussions during the preparation of this work.  BS and JT were supported by NSF DMS-1653264, NSF DMS-2246906, and the Wisconsin Alumni Research Foundation.  JT received additional support from NSF DMS-2037851.  The authors are grateful to the anonymous referee for valuable suggestions.

\appendix  

\section{The Penrose transform}\label{sec:penrose_tools}
The Penrose map is a classical conformal map of Minkowski spacetime~\cite{Pen64}. 
The associated Penrose transform has made previous appearances in sharp restriction theory \cite{GN20, Ne18,Ne22}; see also the recent survey~\cite[§5]{NOST22}.
The purpose of this appendix is to present in self-contained form all the background material on the Penrose transform that is  necessary to treat the Euler--Lagrange equation~\eqref{eq:EulerLagrange}. This  is mostly classical and has already been covered in the aforementioned papers using tools from conformal geometry. Here we shall follow an alternative route that avoids such tools, to the advantage of the more analytically minded reader. 

Given $d\geq 2$, parametrize $\mathbb S^{d}=\{X=(X_0,X_1, \ldots, X_d)\in\R\times\R^d: \sum_{i=0}^d X_i^2=1\}$ in spherical  coordinates, by letting $\vec X=(X_1, \ldots, X_d)$ and
\[X=(X_0,\vec X)=(\cos R, \omega\sin R), \,\,\,R\in[0,\pi],\,\omega\in\mathbb S^{d-1}.\]
On $\mathbb R^d$, introduce polar coordinates $r\ge 0$ and $\omega\in \mathbb S^{d-1}$. We then define the {\it Penrose map} $\mathcal P:\R^{1+d}\to[-\pi,\pi]\times\mathbb{S}^d$ via $\mathcal P(t,r\omega)=(T,\cos R,\omega\sin R)$, where 
\begin{equation} \label{eq:PenroseEquations}
    \begin{split}
        T&=\arctan(t+r)+\arctan(t-r),\\
        R&=\arctan(t+r)-\arctan(t-r), \\
        \omega&=\omega.
    \end{split}
\end{equation}
The inverse map can then be concisely described as follows:  
\begin{equation}\label{eq:inverse_Penrose_Map}
    t\pm r =\tan\left(\frac{T\pm R}{2}\right).
\end{equation}
The range of $\mathcal P$ is often called the \textit{Penrose diamond}, given by
\begin{equation}\label{eq:PenroseDiamond}\notag
    \mathcal{P}(\mathbb R^{1+d})=\{ (T, \cos R, \omega \sin R)\ :\ 0\le R<\pi-\lvert T\rvert \}.
\end{equation}
We define the \textit{conformal factor}\footnote{For the link with conformal geometry, see~\cite[Appendix A.4]{Hor94}.}
\begin{equation}\label{eq:Omega}
\Omega:=\frac{2}{(1+(t+r)^2)^{\frac12}(1+(t-r)^2)^{\frac12}}=\cos T+\cos R.
\end{equation}
The pushforward via $\mathcal{P}$ of the volume element $\d t\d x$ of $\mathbb{R}^{1+d}$ can then be conveniently expressed as
\begin{equation}\label{eq:pushfoward_measure}
    \Omega^{d+1}\d t\d x = \d T\d\sigma,
\end{equation}
where $\d\sigma$ denotes the usual surface measure on $\mathbb S^d$.
To verify~\eqref{eq:pushfoward_measure}, note that the measure $\d T\d\sigma$ satisfies the recursive relation
\begin{equation}\label{eq:factorize_dsigma}\notag
    \d T\d\sigma_{\mathbb S^d}(\cos R, \omega\sin R)=(\sin R)^{d-1}\d T\d R\d\sigma_{\mathbb S^{d-1}}(\omega),
\end{equation}
from which~\eqref{eq:pushfoward_measure} follows directly.

The following result describes the effect of the Penrose map on the d'Alembertian. 
\begin{lemma}\label{lem:main_penrose}
    Given $U\in C^\infty(\mathcal P(\mathbb R^{1+d}))$, define 
    \begin{equation}\label{eq:Penrose_Transform_Spacetime}
        \begin{array}{cc}
            u(t, x):=\Omega^{\frac{d-1}{2}}U(T, X), & \text{ where }(T, X)=\mathcal P(t, x).
        \end{array}
    \end{equation}
    Then $u\in C^\infty(\mathbb R^{1+d})$ and 
    \begin{equation}\label{eq:DAlembertPushforward}
        \left(\partial_t^2-\Delta\right) u(t, x)= \Omega^{\frac{d+3}{2}}\cdot \left( \partial_T^2-\Delta_{\mathbb S^d}+\frac{(d-1)^2}{4}\right) U (T, X).
    \end{equation}
\end{lemma}
\begin{proof}
    We need to prove the following identity:
    \begin{equation}\label{eq:EquivDAlembertPushforward}
        (\partial^2_t-\Delta)\left[(\Omega^{\frac{d-1}{2}} U)|_{(T, X)=\mathcal{P}(t, x)}\right]=\Omega^\frac{d+3}{2}\cdot \left(\partial^2_T - \Delta_{\mathbb S^d} +\frac{(d-1)^2}{4}\right)U.
    \end{equation}
     Since $\mathcal{P}$ leaves $\omega$ invariant, no generality is lost in assuming $U=U(T, R)$, or equivalently $u=u(t, r)$. To begin the proof of~\eqref{eq:EquivDAlembertPushforward}, we first compute the expression of the operator $\partial_t^2-\Delta=\partial_t^2- \partial^2_r -\frac{d-1}{r}\partial_r$ in the coordinates $(T, R)$, and claim that
\begin{equation}\label{eq:DAlembertInPenrose}
        \partial_t^2- \partial^2_r -\frac{d-1}{r}\partial_r= \Omega^2\cdot(\partial_T^2-\partial_R^2)-(d-1)\Omega\cdot\left( \frac{1+\cos R \cos T}{\sin R} \partial_R -\sin T\partial_T\right).
    \end{equation}
    This is most easily verified by first observing that $\partial_t^2-\partial_r^2=4\partial_{t+r}\partial_{t-r}$, so by~\eqref{eq:inverse_Penrose_Map} and the fact that $\Omega=\cos T+\cos R=2\cos\left(\frac{T+R}{2}\right)\cos\left(\frac{T-R}{2}\right)$,
    \begin{equation*}
        \begin{split}
            4\partial_{t+r}\partial_{t-r}=16\cos^2\left(\frac{T+R}{2}\right)\cos^2\left(\frac{T-R}{2}\right)\partial_{T+R}\partial_{T-R} 
        =\Omega^2\cdot(\partial_T^2-\partial_R^2).
        \end{split}
    \end{equation*}
    To handle the remaining term $\frac{d-1}{r}\partial_r$, we use~\eqref{eq:inverse_Penrose_Map} to compute  
    \begin{equation*}
        \frac{d-1}{r}=\frac{2(d-1)}{\tan\left(\frac{T+R}{2}\right)-\tan\left(\frac{T-R}{2}\right)} =\frac{(d-1)\Omega}{\sin R};
    \end{equation*}
    %
    %
     on the other hand, by~\eqref{eq:PenroseEquations},
    \begin{equation*}\label{eq:partial_r}
        \begin{split}
            \frac{\partial}{\partial r}=\frac{\partial T}{\partial r}\frac{\partial}{\partial T} + \frac{\partial R}{\partial r}\frac{\partial}{\partial R}    
            =-\sin R \sin T \frac{\partial}{\partial T} +(1+\cos R\cos T)\frac{\partial}{\partial R},
        \end{split}
    \end{equation*}
    from which~\eqref{eq:DAlembertInPenrose} follows at once.
    %
    %
    %
    %
    %

    To complete the proof of~\eqref{eq:EquivDAlembertPushforward},  we apply the operator on the right-hand side of~\eqref{eq:DAlembertInPenrose} to the function $\Omega^\frac{d-1}{2}U=(\Omega^\frac{d-1}{2}U)(T, R)$. A lenghty but routine computation reveals that
    \begin{equation}\label{eq:tedious_but_straightforward}
        \begin{split}
            &\Omega^2\cdot(\partial_T^2-\partial_R^2)(\Omega^\frac{d-1}{2}U)-(d-1)\Omega\cdot\left( \tfrac{1+\cos R \cos T}{\sin R} \partial_R -\sin T\,\partial_T\right)(\Omega^\frac{d-1}{2}U) \\
            &= \Omega^\frac{d+3}{2}\cdot\left(\partial_T^2 - \partial_R^2 -(d-1)\cot R \,\partial_R +\tfrac{(d-1)^2}{4}\right) U,
        \end{split}
    \end{equation}
    where $\cot$ denotes the cotangent function. By the assumption $U=U(T, R)$, the spherical Laplacian reads 
    \begin{equation*}
        \Delta_{\mathbb S^d}U=(\sin R)^{1-d}\partial_R((\sin R)^{d-1}\partial_R U) = \partial_R^2U +(d-1)\cot R\,\partial_R U, 
    \end{equation*}
    and so we see that the right-hand side of~\eqref{eq:tedious_but_straightforward} coincides with the left-hand side of~\eqref{eq:EquivDAlembertPushforward}. This concludes the proof of the lemma.
    %
    %
\end{proof}

We wish to apply Lemma \ref{lem:main_penrose} to the half-wave propagator $u=e^{itD}g$, defined in \eqref{eq:HalfWaveProp}. Note that $g=u|_{t=0}$. From~\eqref{eq:PenroseEquations} it is immediate that $t=0$ if and only if $T=0$, in which case 
\begin{equation}\label{eq:stereo_projection}
    \begin{array}{cc}\displaystyle
        \cos R=\cos(2\arctan r)=\frac{1-r^2}{1+r^2} \,\, \text{ and }\sin R=\frac{2r}{1+r^2}.
    \end{array}
\end{equation}
These equations coincide with those for the stereographic projection of $\mathbb R^d$ onto $\mathbb S^d\setminus \{(-1, \vec{0})\}$, which can  be  rewritten as $x=\vec{X}/(1+X_0)$.
Next we define
\begin{equation}\label{eq:OmegaZero}   \notag
    \Omega_0:=\Omega\lvert_{t=0}=\frac2{1+r^2}=1+\cos R,
\end{equation}
and  introduce the spherical fractional operator
\begin{equation}\label{eq:SphFracOp}
    D_{\mathbb S^d}:=\left(-\Delta_{\mathbb S^d}+\frac{(d-1)^2}{4}\right)^{\frac12}.
\end{equation} 
We recognize $D_{\mathbb S^d}^2$ as the spatial part of the spherical d'Alembertian on the right-hand side of~\eqref{eq:DAlembertPushforward}.  
    The action of $D_{\mathbb S^d}$ on a  spherical harmonic $Y_\ell$ of degree $\ell$ on $\mathbb S^d$  is 
    \begin{equation}\label{eq:D_S_definition}
    D_{\mathbb S^d}Y_\ell=\left(\ell+\frac{d-1}{2}\right)Y_\ell,
\end{equation}
simply because $-\Delta_{\mathbb S^d}Y_\ell=\ell(\ell+d-1)Y_\ell$.
Thus we  define the propagator $e^{iTD_{\mathbb S^d}}$ via
\begin{equation}\label{eq:SHexpansion}
e^{iTD_{\mathbb S^d}}Y_\ell(X)=e^{iT(\ell+\frac{d-1}2)} Y_{\ell}(X).
\end{equation}

Having settled these preliminaries, we proceed to define the Penrose transform.
\begin{definition}\label{def:Penrose_Transform}
    Given $G\in C^\infty(\mathbb S^d)$, define $g\in C^\infty(\mathbb R^d)$ by
    \begin{equation}\label{eq:PenroseTransformSpaceonly}
       g(r\omega)=\Omega_0^{\frac{d-1}{2}}G(X)  ,
    \end{equation}
    where $X=(\cos R, \omega \sin R)$ with $(\cos R, \sin R)$ given by~\eqref{eq:stereo_projection}. We call $g$ the \textit{Penrose transform} of $G$.
\end{definition}
\begin{remark}\label{rem:foschian_is_constant}
    If $G$ is the constant function $\mathbf 1$ on $\mathbb S^d$, then its Penrose transform is precisely the function $g_\star$ defined in~\eqref{eq:Foschian_Physical}.
\end{remark}
\noindent Note that~\eqref{eq:PenroseTransformSpaceonly} coincides with the evaluation of~\eqref{eq:Penrose_Transform_Spacetime} at $t=0$. The following result is known in conformal geometry as an \textit{intertwining law} \cite[eq.~(1.1)]{Gon13}.
\begin{lemma}\label{lem:morpurgo}
Let $g$ and $G$ be related as in Definition \ref{def:Penrose_Transform}. Then
   \begin{equation}\label{eq:morpurgo}
     Dg(x)=\Omega_0^{\frac{d+1}2} D_{\mathbb S^d}G (X).
    \end{equation}
\end{lemma}
\begin{proof}
     Letting $h=Dg$, we will prove the following equivalent version of~\eqref{eq:morpurgo}:
    \begin{equation}\label{eq:inverse_morpurgo}\notag
         D^{-1}h(x)=\Omega_0^{\frac{d-1}{2}}D_{\mathbb S^d}^{-1}\tilde H(X), 
    \end{equation}
    where $\tilde H(X)=\Omega_0^{-\frac{d+1}{2}}(x) h(x)$, with $X=(\cos R, \omega \sin R)$ and $(\cos R, \sin R)$ given by the stereographic projection~\eqref{eq:stereo_projection}. It is classical~\cite[Def.~2.11]{Kwa15} that
\begin{equation}\label{eq:fractional_integral}
        D^{-1}h(x)=c_d \int_{\mathbb R^d} \frac{h(z)}{\lvert x-z\rvert^{d-1}}\, \d z, \,\,\,\text{where } c_d=\frac{\Gamma\left(\frac{d-1}{2}\right)}{2\pi^\frac{d+1}{2}}.
    \end{equation}
   Letting  $X$ and $Z$ denote the stereographic projection~\eqref{eq:stereo_projection} of $x$ and $z$, respectively, we note that 
    \begin{equation}\label{eq:ChordalDistanceFormula}\notag
        \begin{array}{cc}
            \Omega_0^d\, \d z= \d\sigma(Z), &  \Omega_0(x)\Omega_0(z)\lvert x-z\rvert^2=\lvert X-Z\rvert^2;
        \end{array}
    \end{equation}
    see \cite[§4.4]{LiLoBook}. The right-hand side of~\eqref{eq:fractional_integral} thus equals 
    \begin{equation}\label{eq:fractional_integral_changed}\notag
        c_d\Omega_0(X)^\frac{d-1}{2} \int_{\mathbb S^d} \frac{ \tilde{H}(Z)}{\lvert X-Z\rvert^{d-1}}\, \d\sigma(Z),
    \end{equation}
    and so it suffices to prove that 
    \begin{equation}\label{eq:fractional_integral_remains}
        c_d\int_{\mathbb S^d}\frac{ \tilde{H}(Z)}{\lvert X-Z\rvert^{d-1}}\, \d\sigma(Z) = D_{\mathbb S^d}^{-1}\tilde{H}(X).
    \end{equation}
    It is enough to verify~\eqref{eq:fractional_integral_remains} for $\tilde{H}=Y_\ell$, a spherical harmonic of degree $\ell\geq 0$. From~\eqref{eq:D_S_definition} it follows that $D_{\mathbb S^d}^{-1} Y_\ell = (\ell+\frac{d-1}{2})^{-1} Y_\ell$. Letting $t=X\cdot Z$, we have $\lvert X-Z\rvert^{1-d}=2^{\frac{1-d}{2}}(1-t)^{\frac{1-d}{2}}$. By the theorem of Funk--Hecke~\cite[Ch.~1, §4]{Mu98}, 
    \begin{equation}\label{eq:funk_hecke_application}\notag
        \begin{split}
            c_d\int_{\mathbb S^d}& \frac{Y_\ell(Z)}{\lvert X-Z\rvert^{d-1}}\, \d\sigma(Z)  = \lambda_{\ell, d} Y_\ell(X), \text{ where}\\ 
            \lambda_{\ell, d}&:=c_d2^{\frac{1-d}{2}}\lvert\mathbb S^{d-1}\rvert \int_{-1}^1 \frac{C_\ell^\frac{d-1}{2}(t)}{C_\ell^\frac{d-1}{2}(1)}(1-t)^{\frac{1-d}{2}}(1-t^2)^\frac{d-2}{2}\, \d t,
        \end{split}
    \end{equation}
    and $C_\ell^\frac{d-1}{2}$ denotes the Gegenbauer polynomial introduced in~\eqref{eq:Gegenbauer}. The proof will  be complete once we show that $\lambda_{\ell, d}=(\ell+\frac{d-1}{2})^{-1}$. Noting that $c_d\lvert\mathbb S^{d-1}\rvert=\Gamma(\tfrac{d-1}{2})/(\Gamma(\tfrac d 2)\pi^\frac12)$ and applying  Lemma~\ref{lem:Rodrigues}, we compute:
    \begin{equation}\label{eq:compute_via_rodrigues}\notag 
        \begin{split}
            \lambda_{\ell, d}&= \frac{2^{\frac{1-d}2-\ell} \Gamma(\tfrac{d-1}2)}{\pi^\frac12\Gamma(\ell+\tfrac d 2)}\int_{-1}^1(1-t^2)^{\ell+\frac{d-2}{2}}\frac{\d^\ell}{\d t^\ell} (1-t)^{\frac{1-d}{2}}\, \d t \\
            &=\frac{2^{\frac{1-d}2-\ell} \Gamma(\tfrac{d-1}2)}{\pi^\frac12\Gamma(\ell+\tfrac d 2)}\left(\frac{d-1}{2}\right)^{(\ell)}\int_{-1}^1(1-t^2)^{\ell+\frac{d-2}{2}}(1-t)^{\frac{1-d}{2}-\ell}\, \d t\\
            &=\frac{ \Gamma(\tfrac{d-1}2)}{\pi^\frac12\Gamma(\ell+\tfrac d 2)}\left(\frac{d-1}{2}\right)^{(\ell)}\mathrm{B}\left(\ell+\frac{d}{2}, \frac12\right) \\
            &= \frac{\Gamma(\tfrac{d-1}{2})}{\Gamma(\ell+\tfrac{d+1}{2})}\left(\frac{d-1}{2}\right)^{(\ell)}=\left(\ell+\frac{d-1}{2}\right)^{-1},
        \end{split}
    \end{equation}
    where we have used the notation $a^{(\ell)}=\prod_{k=1}^\ell (a+k-1)$ for the rising factorial and the well-known formula $\mathrm{B}(x, y)=\Gamma(x)\Gamma(y)/\Gamma(x+y)$ for the Beta function. This concludes the proof of the lemma.
\end{proof}
\noindent We can finally state and prove the key property of the Penrose transform. 
\begin{proposition}
\label{thm:main_penrose}
    For every $(t, x)\in\mathbb R^{1+d}$,
    \begin{equation}\label{eq:propagator_equation}\notag
        e^{itD}g(x)=  \Omega^{\frac{d-1}{2}}e^{iT D_{\mathbb S^d}} G(X),
    \end{equation}
    where $(T, X)=\mathcal P(t, x)$.
\end{proposition}
\noindent In light of~\eqref{eq:SHexpansion}, $e^{iTD_{\mathbb S^d}} G(X)$ is defined for every $(T, X)\in \mathbb R\times \mathbb S^d$, but it is related to $e^{itD}g(x)$ only when $(T, X)\in \mathcal{P}(\mathbb R^{1+d})$.   
\begin{proof}[Proof of Proposition~\ref{thm:main_penrose}] 
Consider the initial value problem 
\begin{equation}\label{eq:curved_wave_initial_value}
    \begin{array}{cc}
        \partial_T^2 V=\Delta_{\mathbb S^d} V-\frac{(d-1)^2}{4}V\text{ on }\mathcal P(\mathbb R^{1+d}), & (V, \partial_T V)|_{T=0}=(F, iD_{\mathbb S^d} F),
    \end{array}
\end{equation}
for an arbitrary initial datum $F\in C^\infty(\mathbb S^d)$. Any two solutions to~\eqref{eq:curved_wave_initial_value} with the same initial datum $F$ must coincide on $\mathcal{P}(\mathbb R^{1+d})$. Indeed, differentiating the \textit{energy}\footnote{Recall that the Penrose diamond $\mathcal{P}(\mathbb R^{1+d})$ is described via $0\le R< \pi-\lvert T\rvert$.}
\begin{equation*}
    E(T):=\int_0^{\pi-|T|}\!\!\! \int_{\mathbb S^{d-1}}\! \left( (\partial_T V)^2 + \lvert\nabla_{\mathbb{S}^{d}} V\rvert^2 +\frac{(d-1)^2}{4} V^2 \right) \, \d\sigma_{\mathbb S^{d-1}} (\sin R)^{d-1} \d R,
\end{equation*}
integrating by parts and then invoking the PDE in~\eqref{eq:curved_wave_initial_value}, we obtain for $T\in(-\pi, \pi)$
\begin{equation*}
    \dot{E}(T)=
    \begin{cases} 
        \displaystyle
        (\sin R)^{d-1}\left.\left( -(\partial_T V-\partial_R V)^2-\frac{\lvert \nabla_{\mathbb S^{d-1}} V\rvert^2}{(\sin R)^2}-\frac{(d-1)^2}{4} V^2\right)\right|_{R=\pi- T}\!\!\!\!\!\!\!& ,T>0, \\
        \displaystyle
        \left.(\sin R)^{d-1}\left( (\partial_T V+\partial_R V)^2+\frac{\lvert \nabla_{\mathbb S^{d-1}} V\rvert^2}{(\sin R)^2}+\frac{(d-1)^2}{4} V^2\right)\right|_{R=\pi+ T}\!\!\!\!\!\!\!& ,T<0.
    \end{cases}
\end{equation*}
If $F=0$, then $E(0)=0$, which then implies $E(T)=0$ for every $T\in (-\pi, \pi)$ by the above identity for $\dot{E}(T)$. The claimed uniqueness follows at once.
Now  let 
\begin{equation*}
    \begin{array}{ccc}
        u(t, x):=e^{itD}g(x),& U(T, X):=\Omega^{\frac{1-d}{2}}u(t, x),& \text{where }(T, X)=\mathcal P(t, x).
    \end{array}
\end{equation*}
It is clear from  definition~\eqref{eq:SphFracOp} of $D_{\mathbb S^d}$ that $V(T, X)=e^{iTD_{\mathbb S^d}} G(X)$ satisfies~\eqref{eq:curved_wave_initial_value} with initial data $(G, iD_{\mathbb S^d} G)$. We claim that $U$ also solves~\eqref{eq:curved_wave_initial_value} with the same initial data as $V$. Once this is proved, the aforementioned uniqueness will imply that $U$ and $V$ must agree on $\mathcal P(\mathbb R^{1+d})$, completing the proof of Proposition~\ref{thm:main_penrose}. 
To verify the claim, we start by noticing that $u$ solves
\begin{equation}\label{eq:wave_initial_value}\notag
    \begin{array}{ccc}
        \partial_t^2 u=\Delta u, & u|_{t=0}=g,& \partial_t u|_{t=0}=i Dg,
    \end{array}
\end{equation}
and so the first identity in~\eqref{eq:curved_wave_initial_value} is an immediate consequence of Lemma~\ref{lem:main_penrose}. The fact that $U|_{T=0}=G$ is also immediate, as we remarked right after Definition~\ref{def:Penrose_Transform}. To check that $\partial_T U|_{T=0}=iD_{\mathbb S^d} G$, note that $\partial_T \Omega|_{T=0}=0$ and $\partial_T|_{T=0}=\Omega_0^{-1}\partial_t|_{t=0}$, which respectively follow from~\eqref{eq:Omega}, and~\eqref{eq:inverse_Penrose_Map} together with the chain rule. Lemma~\ref{lem:morpurgo} then implies
\begin{equation}\label{eq:check_iDS_penrose_proof}\notag
    \partial_T U|_{T=0}=\Omega_0^{-\frac{d+1}{2}}\partial_t u|_{t=0}= i\Omega_0^{-\frac{d+1}{2}}Dg=iD_{\mathbb S^d}G, 
\end{equation}
which completes the proof of the proposition.
\end{proof}

\subsection{Application to the Euler--Lagrange equation}\label{apsec:A1} 
In this section, we establish formula~\eqref{eq:LHSPenrose}, which is a consequence of the following result. Recall  $g_\star=\Omega_0^{\frac{d-1}2}$.
\begin{lemma}\label{lem:LHSPenroseAppendix}
    Let $g$ be the Penrose transform of $G$. Then
    \begin{equation}\label{eq:LHSPenroseAppendix}
        \int_{\R^{1+d}} |e^{itD} g_\star|^{p-2} (\overline{e^{itD} g_\star}) (e^{itD} g) \,\d t\d x = \frac12 \int_{-\pi}^\pi \int_{\mathbb S^d} e^{-iT\frac{d-1}{2}} (e^{iTD_{\mathbb S^d}} G)\lvert \Omega\rvert^{\gamma_p} \d \sigma\d T.
    \end{equation}
\end{lemma}
\begin{proof}
The Penrose transform of the constant function $G_\star={\bf 1}$ is $g_\star$, and $e^{iTD_{\mathbb S^d}}\mathbf 1 = e^{iT\frac{d-1}{2}}$. By Proposition~\ref{thm:main_penrose} and a change of variables (recall \eqref{eq:pushfoward_measure}),
\begin{equation}\label{eq:EL_first_computation}
    \begin{split}
        &\int_{\R^{1+d}} |e^{itD} g_\star|^{q-2} \overline{e^{itD} g_\star} e^{itD} g \,\d t\d x \\&=
        \int_{\mathcal{P}(\R^{1+d})}| \Omega^{\frac{d-1}2} e^{iTD_{\mathbb S^d}} {\bf 1}|^{q-2} (\overline{ \Omega^{\frac{d-1}2}e^{iTD_{\mathbb S^d}} {\bf 1}})  ( \Omega^{\frac{d-1}2} e^{iTD_{\mathbb S^d}} G) \Omega^{-(d+1)}\d T\d \sigma \\
        &=\int_{\mathcal{P}(\R^{1+d})}e^{-iT\frac{d-1}{2}} (e^{iTD_{\mathbb S^d}} G)|\Omega|^{\gamma_p} \d T\d \sigma,
    \end{split}
\end{equation}
where $\gamma_p=\frac{d-1}2q-(d+1)=(d+1)(\frac{p'}2-1)$ was defined in \eqref{eq:alpha_p_q_first}. This is still not the desired~\eqref{eq:LHSPenroseAppendix}, since the last integral in~\eqref{eq:EL_first_computation} is over the Penrose diamond $\mathcal P(\mathbb R^{1+d})$ and not the product space $[-\pi, \pi]\times \mathbb S^d$. 
To remedy this, observe that 
\begin{equation}\label{eq:invU}\notag
    e^{i(T+\pi)D_{\mathbb S^d}}G(-X)=e^{i\frac{d-1}{2}\pi}e^{iTD_{\mathbb S^d}}G(X),
\end{equation}
for every $(T,X)\in[-\pi,\pi]\times\mathbb S^{d}$. This follows from ~\eqref{eq:SHexpansion} and  the fact that $Y_{\ell}(-X)=(-1)^\ell Y_{\ell}(X)$. Recalling  $\Omega=\cos T+ \cos R$, we conclude that the integrand $V(T,X):=e^{-iT\frac{d-1}{2}} (e^{iTD_{\mathbb S^d}} G)\lvert\Omega\rvert^{\gamma_p}$ satisfies
\begin{equation}\label{eq:crucial_symmetry}
    V(T+\pi,-X)=V(T,X), \text{ for every }(T,X)\in[-\pi,\pi]\times\mathbb S^d.
\end{equation}
As noticed in \cite[Lemma 3.6]{Ne18}, this symmetry implies $\int_{\mathcal P(\mathbb R^{1+d})} V=\frac12 \int_{-\pi}^\pi \int_{\mathbb S^d} V$. Indeed, letting 
$$ 
W(T, R):=\int_{\mathbb S^{d-1}} V(T, \cos R, \omega \sin R)(\sin R)^{d-1}\, \d\sigma_{\mathbb S^{d-1}}(\omega),$$ 
we have that $W(T+\pi, \pi- R)=W(T, R)$, and so 
\begin{equation}\label{eq:integral_symmetries}\notag
    \begin{array}{c}
       \displaystyle \int_{-\pi}^0\int_{\pi+T}^\pi W(T, R)\, \d R\d T= \int_{0}^\pi\int_{0}^{\pi- T}  W(T, R)\, \d R\d T,\\ \displaystyle \int_{-\pi}^0\int_{0}^{\pi+T} W(T, R)\, \d R\d T= \int_{0}^\pi\int_{\pi- T}^{\pi}  W(T, R)\, \d R\d T,
    \end{array}
\end{equation}
as can be seen via the changes of variables $T'=T+\pi$ and $R'=\pi-R$. The claim follows immediately: 
\begin{equation*}
    \begin{split}
        \int_{-\pi}^\pi \int_{\mathbb S^d} V &= \left(  \int_{-\pi}^0\int_{\pi+T}^\pi + \int_{-\pi}^0\int_{0}^{\pi+T} + \int_{0}^\pi\int_{0}^{\pi- T}+\int_{0}^\pi\int_{\pi- T}^{\pi} \right) W(T, R)\, \d R\d T \\ 
        &= 2\left(  \int_{-\pi}^0\int_{0}^{\pi+T} + \int_{0}^\pi\int_{0}^{\pi- T}\right) W(T, R)\, \d R\d T=2 \int_{\mathcal P(\mathbb R^{1+d})} V.
    \end{split}
\end{equation*}
This establishes~\eqref{eq:LHSPenroseAppendix} and concludes the proof of the lemma.
\end{proof}

\subsection{Symmetry}\label{apsec:tangent} 
In this final section, we elaborate on the symmetry considerations underlying formula \eqref{eq:intro_symmetry} and Remark~\ref{rem:no_low_degree_spherical_harmonics}. 

 We make two remarks regarding formula~\eqref{eq:intro_symmetry} that we now recall,
\begin{equation}\label{eq:Lq_Norm_Symmetric_Only_SteinTomas}\tag{\ref{eq:intro_symmetry}}
       \lVert e^{it D}g \rVert_{L^q(\mathbb R^{1+d})}^q= \frac12 \int_{-\pi}^\pi \int_{\mathbb S^d} \left \lvert e^{iTD_{\mathbb S^d}} G\right\rvert^q \lvert \Omega\rvert^{(d+1)(\frac{p'}{2}-1)} \d \sigma\d T,
\end{equation}
and which is straightforward to prove via the computations in \S\ref{apsec:A1}, which rely on the crucial symmetry \eqref{eq:crucial_symmetry}.
Firstly, the function $\lvert \Omega\rvert^{(d+1)({p'}/{2}-1)}$ is integrable on $[-\pi, \pi]\times \mathbb S^d$ if and only if $p,q$ belong to the conjectural range~\eqref{eq:RCrange}. Indeed, recalling \eqref{eq:Omega}, the  integral on the right-hand side of \eqref{eq:intro_symmetry} becomes
\begin{equation*}
    \int_{-\pi}^\pi\int_0^\pi\lvert \cos T+\cos R\rvert^{(d+1)(\frac{p'}{2}-1)}\left(\int_{\mathbb S^{d-1}}\left \lvert e^{iTD_{\mathbb S^d}} G\right\rvert^q\, \d \sigma_{\mathbb S^{d-1}}\right) (\sin R)^{d-1} \,\d R\, \d T. 
\end{equation*}
The  singularity at $R=\pi-\lvert T\rvert$ is integrable if and only if  ${(d+1)(\frac{p'}{2}-1)}>-1$, or  $1<p<\frac{2d}{d-1}$, as claimed.
Secondly, in the Strichartz case $p=2$ identity \eqref{eq:intro_symmetry}  implies that the left-hand side of~\eqref{eq:cone_restr_recast} remains invariant\footnote{This is {\it not} the case for $p\ne 2$, in light of the symmetry breaker $\lvert\Omega\rvert=\lvert \cos T+ X_0\rvert$.} under the action 
    \begin{equation}\label{eq:G_rotates} 
        G\mapsto G\circ \rho, 
    \end{equation}
where $\rho$ denotes an arbitrary rotation of $\mathbb S^d$. In this case, the right-hand side of~\eqref{eq:cone_restr_recast} is also invariant under~\eqref{eq:G_rotates}. Indeed, by Lemma~\ref{lem:morpurgo} (and $\Omega_0^d\, \d x= \d \sigma$), 
\begin{equation*}
    \lVert \widehat g\rVert_{L^2(\R^d,|\xi|\d\xi)}^2=\int_{\mathbb R^d} \overline{g(x)} Dg(x)\, \d x= \int_{\mathbb S^d} \overline{G(\omega) }D_{\mathbb S^d} G(\omega)\, \d \sigma,
\end{equation*}
which is manifestly invariant under rotations. Thus, for $p=2$,~\eqref{eq:G_rotates} is indeed a hidden symmetry of the Strichartz estimate \eqref{eq:SteinTomas}.

As noted in the introduction, the Fourier extension operator from the cone coincides with the half-wave propagator. On the other hand, if we equip the two-sheeted cone $\{(\tau, \xi)\in \mathbb R^{1+d} : \tau^2=\lvert\xi\rvert^2\}$ with the Lorentz-invariant measure $\ddirac{\tau^2-\lvert\xi\rvert^2}\, \d\tau\d\xi$, the Fourier extension operator then coincides with the propagator for the wave equation, $\partial_t^2u = \Delta u$. Given a solution $u$ to the latter, Lemma \ref{lem:main_penrose} yields the following formula, which is analogous to \eqref{eq:intro_symmetry}:
\begin{equation}\label{eq:full_wave_symmetry}
    \lVert u\rVert_{L^q(\mathbb R^{1+d})}^q=\int_{\mathcal P(\mathbb R^{1+d})}\lvert U\rvert^q\lvert \Omega\rvert^{(d+1)(\frac{p'}{2}-1)} \d \sigma\d T,
\end{equation}
with $\partial_T^2U=\Delta_{\mathbb S^d}U-\frac{(d-1)^2}{4}U$. It turns out that an arbitrary solution to this equation satisfies the crucial symmetry \eqref{eq:crucial_symmetry} if and only if $d$ is an odd number. Consequently, the right-hand side of \eqref{eq:full_wave_symmetry} can be extended to an integral on the Cartesian product $[-\pi, \pi]\times \mathbb S^d$, like in \eqref{eq:intro_symmetry}, only when $d$ is odd. In particular, the right-hand side of \eqref{eq:full_wave_symmetry} is invariant under arbitrary rotations of $\mathbb S^d$ only when $d$ is odd (for $p=2$). This leads to  $\mathfrak F$-functions from Definition \ref{def:Foschian} not being critical points for the Fourier extension inequality from the two-sheeted cone in the Strichartz case $p=2$ when $d$ is even; this fact, which was discovered in \cite{Ne18}, is discussed at length in  \cite[§5.1]{NOST22}.

    We conclude with a discussion of Remark~\ref{rem:no_low_degree_spherical_harmonics}, which applies to general exponents $p$.
The symmetry group $\mathcal S$ from \S\ref{sec:symmetries} consists of multiplication by unimodular complex constants, conic dilations, Lorentz boosts and spacetime translations. We disregard the latter two because they fail to preserve radial symmetry and, in light of \eqref{eq:F_is_Yk},  only consider radial functions. On the other hand, we add  multiplication by positive constants, which is not a symmetry in the sense of \S\ref{sec:symmetries}, but does leave the functional $\Phi_{p, q}$ from~\eqref{eq:Phi_Functional} invariant. 
By such invariance, applying the infinitesimal generators of these symmetries to the function $g_\star$ from \eqref{eq:Foschian_Physical} results in a vector space of test functions that automatically satisfy the Euler--Lagrange equation~\eqref{eq:EulerLagrange}, and are thus unsuitable to disprove it. We now show that this vector space coincides with the Penrose transform of the space of zonal spherical harmonics of degree zero and one.
When applied to $g_\star$, the generators of the aforementioned symmetries form the  vector space 
\begin{equation}\label{eq:tangent_space_physical}
    \mathrm{span}_{\mathbb R}( g_\star, ig_\star, x\cdot \nabla g_\star, iDg_\star),
\end{equation}
corresponding to multiplication by positive constants, unimodular complex constants, conic dilations and time translations. Since $g_\star(x)\cong_d (1+\lvert x\rvert^2)^{\frac{1-d}2}$, we have
\begin{equation*}
    x\cdot \nabla g_\star (x) \cong_d \lvert x \rvert^2 \Omega_0^{\frac{d+1}2}(x),
\end{equation*}
whereas \eqref{eq:D_S_definition} and Lemma~\ref{lem:morpurgo} together imply
\begin{equation*}
    iD g_\star (x)= i\frac{d-1}{2}(1+X_0)^\frac{d+1}{2}.
\end{equation*}
Thus $x\cdot \nabla g_\star(x)\cong_d \Omega_0^{\frac{d-1}{2}}(1-X_0)$ and $iD g_\star (x) \cong_d \Omega_0^{\frac{d-1}{2}} (i + i X_0)$. We conclude that ~\eqref{eq:tangent_space_physical} is the Penrose transform of 
\begin{equation}\label{eq:tangent_space_Penrose}\notag
    \mathrm{span}_{\mathbb R}(1, i , 1-X_0, i+iX_0) = \mathrm{span}_{\mathbb R}(1, i, X_0, iX_0),
\end{equation}
which coincides with the complex vector space of zonal spherical harmonics of degree zero and one, as claimed. \\

\end{document}